\documentclass[onefignum,onetabnum]{siamart190516}

\usepackage{amsfonts}
\usepackage{graphicx}
\usepackage{epstopdf}
\usepackage{algorithmicx, algpseudocode}
\usepackage{mathtools}
\usepackage{commath,nicefrac}

\usepackage{tikz}

\usepackage{bm}
\usepackage{subcaption}

\ifpdf  \DeclareGraphicsExtensions{.eps,.pdf,.png,.jpg}
\else
  \DeclareGraphicsExtensions{.eps}
\fi

\DeclareMathOperator*{\argmax}{arg\,max}
\DeclareMathOperator*{\argmin}{arg\,min}

\def\ddefloop#1{\ifx\ddefloop#1\else\ddef{#1}\expandafter\ddefloop\fi}
\def\ddef#1{\expandafter\def\csname bf#1\endcsname{\ensuremath{\mathbf{#1}}}}
\ddefloop ABCDEFGHIJKLMNOPQRSTUVWXYZabcdefghijklmnopqrstuvwxyz\ddefloop

\def\ddef#1{\expandafter\def\csname v#1\endcsname{\ensuremath{\boldsymbol{#1}}}}
\ddefloop ABCDEFGHIJKLMNOPQRSTUVWXYZabcdefghijklmnopqrstuvwxyz\ddefloop

\def\ddef#1{\expandafter\def\csname v#1\endcsname{\ensuremath{\boldsymbol{\csname #1\endcsname}}}}
\ddefloop {alpha}{beta}{gamma}{delta}{epsilon}{varepsilon}{zeta}{eta}{theta}{vartheta}{iota}{kappa}{lambda}{mu}{nu}{xi}{pi}{varpi}{rho}{varrho}{sigma}{varsigma}{tau}{upsilon}{phi}{varphi}{chi}{psi}{omega}{Gamma}{Delta}{Theta}{Lambda}{Xi}{Pi}{Sigma}{varSigma}{Upsilon}{Phi}{Psi}{Omega}{ell}\ddefloop

\def\ddef#1{\expandafter\def\csname bb#1\endcsname{\ensuremath{\mathbb{#1}}}}
\ddefloop ABCDEFGHIJKLMNOPQRSTUVWXYZ\ddefloop

\newsiamremark{remark}{Remark}
\newsiamremark{hypothesis}{Hypothesis}
\crefname{hypothesis}{Hypothesis}{Hypotheses}
\newsiamthm{claim}{Claim}

\headers{Least-Squares Reduced Basis}
{J.~H.~Chaudhry, L.~N.~Olson, and P.~Sentz}

\title{A Least-Squares Finite Element Reduced Basis Method
  \thanks{Submitted to the editors March 4, 2020\@.
\funding{J. Chaudhry's work is supported by the NSF-DMS 1720402.}}}

\author{Jehanzeb Hameed Chaudhry\thanks{Department of Mathematics, University of New Mexico (\email{jehanzeb@unm.edu}, \url{https://math.unm.edu/\~jehanzeb}).}
\and Luke N. Olson\thanks{Department of Computer Science (\email{lukeo@illinois.edu}, \url{http://lukeo.cs.illinois.edu}).}
\and Peter Sentz\thanks{Department of Computer Science (\email{sentz2@illinois.edu}).}
}

\usepackage{amsopn}

\ifpdf\hypersetup{  pdftitle={A Least-Squares Finite Element Reduced Basis Method},
  pdfauthor={J.~H.~Chaudhry, L.~N.~Olson, and P. Sentz}
}
\fi

\usepackage{cleveref}

\begin{document}

\maketitle

\begin{abstract}
We present a reduced basis (RB) method for parametrized linear elliptic partial differential equations (PDEs) in a least-squares finite element framework.  A rigorous and reliable error estimate is developed, and is shown to bound the error with respect to the exact solution of the PDE, in contrast to estimates that measure error with respect to a finite-dimensional (high-fidelity) approximation.  It is shown that the first-order formulation of the least-squares finite element is a key ingredient.  The method is demonstrated using numerical examples.
\end{abstract}



\section{Introduction}
In this work, we formulate a reduced basis method for the solution of linear elliptic partial differential equations (PDEs) based on the least-squares finite element method (LSFEM).  In many engineering and scientific applications, PDEs often depend on one or more parameters, which reflect either physical properties (e.g.,\ the viscosity of a fluid, the heat conductivity of a medium), source terms and boundary conditions, or the geometry of the domain in which the problem is posed.  In the case of parametrized geometry, transformation techniques~\cite{chaudhry2018efficient, manzoni2016automatic,quarteroni2015reduced} are used to obtain a PDE on a parameter-independent reference domain $\Omega$.  Letting $\vmu$ be a vector containing the relevant parameters, we study linear elliptic PDEs of the form:
\begin{equation}\label{parametrized_equation}
	\mathcal{L}_{\vmu}u_{\vmu} = f_{\vmu},\hspace{7mm} \vx \in \Omega,
\end{equation}
where $\Omega$ is a bounded subset of $\mathbb{R}^d$, $d = 2,3$.
The subscript $\vmu$ conveys the fact that the operator $\mathcal{L}_{\vmu}$ and the functions $u_{\vmu}$ and $f_{\vmu}$ depend on the value of the parameter(s) contained in $\vmu$.  In this work, we consider elliptic problems in \eqref{parametrized_equation}~---~e.g., the Poisson's Equation with different values for the thermal conductivity of a medium, or the Stokes Equations with a varying Reynolds number.

LSFEMs are widely used for the solution of PDEs arising in many applications in science and engineering like fluid flow, transport, hyperbolic equations, quantum chromodynamics, magnetohydrodynamics, biomolecular simulation, plasma, elasticity, liquid crystals etc.~\cite{adler2014error, atherton2012competition, bochev1999analysis, bochev2009least,bond2010first, brannick2010least,cai2004least, de2004least,de2005numerical, heys2007alternative,krause2017adaptive, leibs2015nested, manteuffel1998least}.  LSFEMs are based on minimizing the residual of the PDE in an appropriate norm, and have a number of attractive properties.  The finite element discretization of the weak form yields symmetric positive definite linear systems that are often suitable for optimal multigrid solvers.  Moreover, the bilinear form arising from LSFEM is coercive and continuous, thus allowing flexibility in the choice of finite element (FE) spaces. This is in contrast to a mixed method which requires that the FE spaces satisfy the inf-sup or the  Ladyzhenskaya-Babu\v{s}ka-Brezzi condition~\cite{boffi2013mixed}.  An additional advantage of LSFEMs is that complex boundary conditions may be handled weakly by incorporating them into the definition of the least-squares residual.

Least-squares finite element methods provide a robust and inexpensive \textit{a posteriori} error estimate. This is a crucial ingredient in our approach to constructing a reduced basis method for LSFEMs.  Moreover, while the additional auxiliary variables and resulting large linear systems is a potential drawback to LSFEMs, a reduced basis approach which preserves the accuracy of the full finite element discretization while being inexpensive to compute is especially appealing for this class of discretizations.

In many applications, solutions are computed for a wide range of parameter values (many-query context), or must be computed cheaply following a parameter measurement or estimation (real-time context)~\cite{boyaval2008reduced, grepl2007certified, oliveira2007reduced, rozza2007reduced, veroy2003posteriori}.  In the case of a finite element discretization, a system of linear equations is obtained that involves a large number of unknowns.  If solutions must be obtained quickly or for many parameter sets, the solution of these linear systems is prohibitively expensive.  Reduced basis methods are a form of model order reduction that offers the potential to decrease the dimension of the problem, exploiting the low dimensionality of the solution manifold through parametric dependence~\cite{prud2001reliable}.  As a result, solutions based on the low order representation are constructed with low computational cost.

RB methods are separated into two stages: ``offline'' and ``online''~\cite{deparis2009, hesthaven2014efficient, quarteroni2015reduced, rozza2007reduced}.  During the offline stage, a set of representative solutions is constructed by sampling the parameter domain and computing high dimensional finite element solutions called full-order model (FOM) solutions or snapshots.  Two standard approaches for the offline basis construction include Proper Orthogonal Decomposition (POD)~\cite{liang2002proper,volkwein2011model} and greedy sampling methods~\cite{hesthaven2014efficient}.  Greedy sampling methods often lead to a more computationally efficient offline stage and are used in numerous applications~\cite{deparis2009, grepl2005, hesthaven2014efficient, huynh2007reduced, quarteroni2015reduced}.  This work is thus restricted to reduced basis methods with a greedy sampling procedure.  Details of POD applied to parametrized elliptic systems is found in~\cite{kahlbacher2007galerkin}.\par
During the online stage, the previously constructed reduced basis is used to generate an inexpensive yet accurate solution for an estimated or measured set of parameters.  The accuracy of this solution strongly depends on the sampling strategy and as well as the selection criteria for choosing the reduced basis.

The accuracy of a reduced basis solution is typically measured in reference to a full-order finite element solution~\cite{dihlmann2015,grepl2005,huynh2007reduced,rozza2013reduced}.  The error $\| u_{\vmu}^h - u_{\vmu}^{\text{RB}}\|$ under an appropriate norm is heuristically minimized, where $u_{\vmu}^h$ and $u_{\vmu}^{\text{RB}}$ are the full-order and reduced basis solutions, respectively.  In essence, the full-order FE solution is treated as the exact solution for every parameter value; it is used as the benchmark for accuracy of the reduced order solution.
However, the accuracy of the full order finite element solution is itself heavily dependent on the value of parameters for certain problems, resulting in an error estimate for the reduced basis solution that is often overly optimistic.  In this article, a sharp error estimate with respect to the \textit{exact} solution of the PDE is used in the construction of the reduced basis during the offline stage. This error estimate is provided by the relaxed smoothness requirements afforded by a first-order formulation, as well as a posteriori error estimate provided naturally by the LSFEM, and is inexpensive to compute, and provides an attractive feature of a LSFEM-based RB method.

To demonstrate the utility of measuring the accuracy of the reduced basis solution in terms of the exact solution, we consider a variable coefficient Poisson's problem, see  \S~\ref{results-thermal-1} for the detailed setup.  The problem is dependent on a single parameter $\mu \in [10^{-1}, 10^1]$, which represents the thermal conductivity of one-half of an inhomogeneous material. The solution is benign for $\mu = 1$, but features a discontinuous gradient for all other values.  Thus, high accuracy requires a very fine mesh.

The left plot of Figure~\ref{intro-plot} shows the error\footnote{The error is measured in the $H(\text{div})\times H^1$-norm, which is the appropriate norm for the least-squares setup for the Poisson's problem. See \S~\ref{results-thermal-1} for details.} between a discrete solution and an ``true'' solution for different values of the parameter $\mu$.  The discrete solution $u_{\mu}^h$ is computed on a mesh with 1,065 degrees of freedom and a reference or ``true''  solution $u_{\mu}^e$ is computed on a mesh with 122,497 degrees of freedom.  The error is particularly large for $\mu = 10^{-1}$.  A reduced basis solution $u_{\mu}^{RB}$ is constructed on the same mesh as $u_{\mu}^h$; the right plot of Figure~\ref{intro-plot} shows the error of this reduced-order solution with respect to both the reference solution and the full-order solution.
\begin{figure}[ht]
\centering
\begin{subfigure}{.49\textwidth}
	\centering
	\includegraphics[scale=0.35]{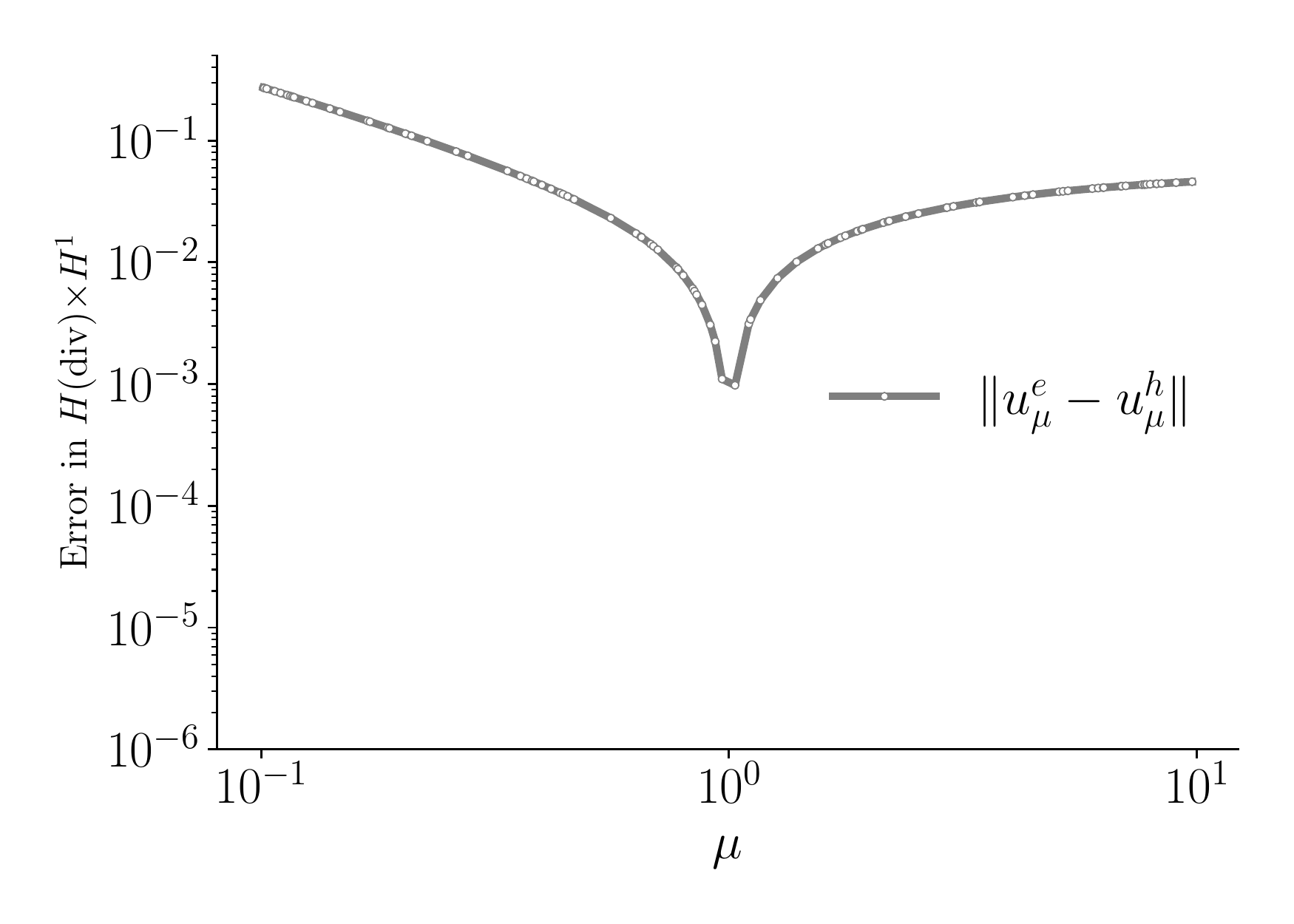}
  \caption{Full-order error.}\label{intro-left}
\end{subfigure}\hfill
\begin{subfigure}{.49\textwidth}
	\centering
	\includegraphics[scale=0.35]{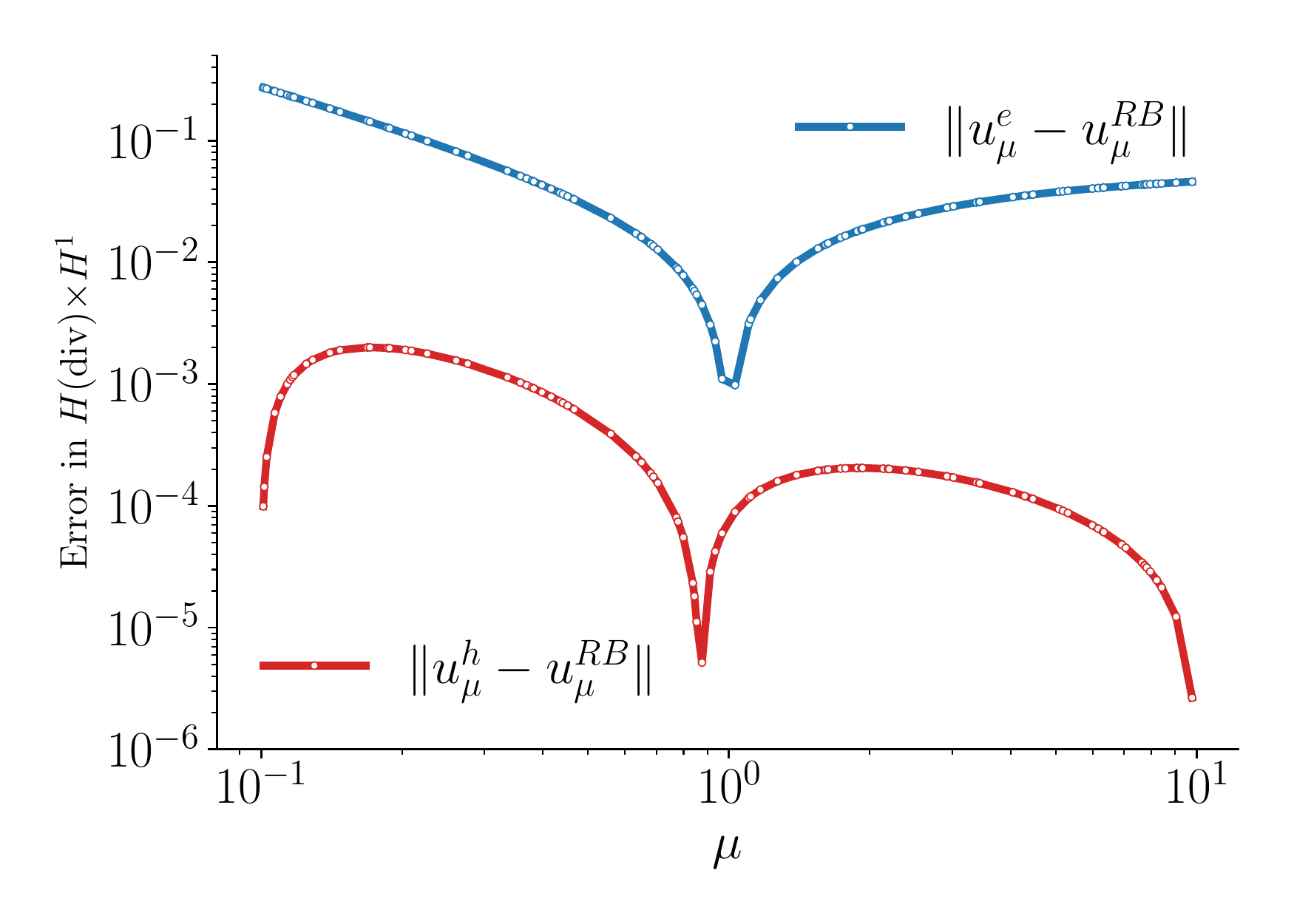}
  \caption{Reduced basis error.}\label{intro-right}
\end{subfigure} \caption{The $H(\text{div})\times H^1$-norm of the error between a full-order solution and a reference solution (left), and the error in the reduced-basis solution with respect to both the full-order solution and the reference (right).}\label{intro-plot}
\end{figure}

In Figure~\ref{intro-plot}, the error between the reduced-basis solution and the reference solution is four orders of magnitude greater than the error with respect to the discrete solution.  Thus, a sharp, rigorous error bound based on $\|u_{\mu}^{RB} - u_{\mu}^h\|$ would significantly underestimate the error with respect to the true solution.  For this reason, an error estimate is not reliable without first ensuring that the full-order solution is sufficiently accurate.

The paper is organized as follows.  In \S~\ref{sec:parameqs}, we describe parametrized equations in a Hilbert space setting, and describe a rigorous error estimate for approximate solutions.  In \S~\ref{sec:ls}, we introduce least-squares finite element methods and how they fit into the abstract Hilbert space context.  In \S~\ref{sec:rbm}, we review standard reduced basis methods, in particular, those based on residual error indicators and greedy-sampling.  In \S~\ref{sec:lsrbm}, we propose a LSFEM-based reduced basis method and in \S~\ref{sec:numerics} we provide several numerical examples.  \S~\ref{sec:conc} consists of conclusions and possibilities for future work.

\section{Parameterized Equations and Error Bounds}\label{sec:parameqs}

In this section we set up the parameterized equations in a Hilbert space
setting.  In \S~\ref{sec:errorbounds} we discuss error bounds in this
context and in \S~\ref{sec:appliedtopoisson} we detail issues that arise
when considering elliptic problems in a standard Galerkin setting.

Let $X$ and $Y$ be Hilbert spaces, and let $\mathcal{D}$ be a compact subset of $\mathbb{R}^P$, with $P \geq 1$.  For any $\vmu \in \mathcal{D}$, we assume the existence of a linear operator
\begin{equation}
	\mathcal{L}_{\vmu}: X \to Y.
\end{equation}
For a fixed $f_{\vmu} \in Y$, we seek $u_{\vmu} \in X$ that satisfies
\begin{equation}\label{strong_form}
	\mathcal{L}_{\vmu}u_{\vmu} = f_{\vmu}.
\end{equation}

We further assume that for any $\vmu \in \mathcal{D}$ there exists a parameter dependent \textit{coercivity} constant $\alpha(\vmu)$ and a \textit{continuity} constant $\gamma(\vmu)$ with $0 < \alpha(\vmu)\leq \gamma(\vmu) < \infty$ such that
\begin{equation}\label{energy_balance}
	\alpha(\vmu)\|v\|_X^2 \leq \|\mathcal{L}_{\vmu}v\|_Y^2\leq \gamma(\vmu)\|v\|_X^2, \hspace{6mm} \forall v \in X.
\end{equation}
That is, $\mathcal{L}_{\vmu}$ and its inverse are bounded.

In order to approximate $u_{\vmu}$, we introduce a finite-dimensional subspace $X^h \subset X$ and seek a function  $u^h_{\vmu} \in X^h$.
  The subspace $X^h$
  may correspond to any general discretization procedure, e.g., finite differences, finite elements, or from a reduced order model.

For a particular parameter $\vmu$, we define the error to be
\begin{equation}\label{error_definition}
	e^h_{\vmu} \coloneqq u_{\vmu} - u_{\vmu}^h,
\end{equation}
which is a measure of the quality of this approximation.
Developing a rigorous and strict upper bound for the norm of the error  $\|e^h_{\vmu}\|_X$ is important for assessing the quality of the numerical approximation.  Likewise,
the residual is defined as
\begin{equation}\label{residual_definition}
r_{\vmu}^h \coloneqq f_{\vmu} - \mathcal{L}_{\vmu}u_{\vmu}^h,
\end{equation}
and we see that $e_{\vmu}^h$ satisfies the error equation
\begin{equation}\label{error_equation}
	\mathcal{L}_{\vmu}e_{\vmu}^h = r_{\vmu}^h.
\end{equation}

\subsection{Error Bounds}\label{sec:errorbounds}

Our approach to developing rigorous upper bounds on the error is to begin with~\eqref{energy_balance}~and~\eqref{error_equation}, which leads to

\begin{align}\label{loose_bound}
\begin{split}
\alpha(\vmu)\|e_{\vmu}^h\|_X^2 & \leq \|\mathcal{L}_{\vmu}e_{\vmu}^h\|_Y^2 = \|r_{\vmu}^h\|_Y^2,\\
\Rightarrow \|e_{\vmu}^h\|_X   & \leq \frac{\|r_{\vmu}^h\|_Y}{\sqrt{\alpha(\vmu)}}.
\end{split}
\end{align}
Unfortunately, this upper bound proves to be extremely pessimistic, especially for problems for which the coercivity constant $\alpha(\vmu)$ is relatively small, a common scenario.  This is illustrated with a simple finite dimensional example.\par
Let $X = Y = \mathbb{R}^n$ under the standard Euclidean norm.  Consider the operator $A : X \to Y$ represented by the matrix
\begin{equation}
	A =
  \left[
  \begin{array}{rrrrrr}
  2  & -1 & \        & \      & \      & \ \\
	-1 & 2  & -1       & \      & \      & \ \\
	\  & -1 & 2        & -1     & \      & \ \\
	\  & \  & \ \ddots & \ddots & \ddots & \ \\
	\  & \  & \        & -1     & 2      & -1\\
	\  & \  & \        & \      & -1     & 2
  \end{array}
  \right],
\end{equation}
which is positive-definite with smallest eigenvalue $\lambda_1 = 4\sin^2\left(\frac{\pi}{2(n+1)}\right)$.
In addition consider the right-hand side $f = [1,0,\dots, 0, 1]^T$, which yields a solution to $Au = f$ of $u = [1,1,\dots,1]^T$.  Then, consider the perturbation $\hat{u} \in \mathbb{R}^n$ given by
\begin{equation}
	\hat{u}_i = 1 + \frac{(-1)^i}{n}.
\end{equation}
The error in this case is given by $\|u - \hat{u}\|_X = \frac{1}{\sqrt{n}}$ and the residual by $\|r\|_X = \|f - Au\|_X = \frac{\sqrt{16n - 14}}{n}$.  Thus, both the error and residual converge to zero as $n \to \infty$.  However, the ratio
\begin{equation}
	\frac{\|r\|_X}{\sqrt{\alpha}} = \frac{\|r\|_X}{\sqrt{\lambda_1}} = \frac{\sqrt{16n-14}}{2n\sin\left(\frac{\pi}{2(n+1)}\right)} > \frac{4\sqrt{n-1}}{\pi}
\end{equation}
is unbounded for large $n$.  The error $u - \hat{u}$ has no component in the span of the eigenvector of $A$ corresponding to $\lambda_1$.  Thus, reflecting on~\eqref{loose_bound}, the ratio of the residual to the square root of the coercivity constant is not an accurate predictor of the norm of the error.

As a consequence, our goal is to improve the error bound in~\eqref{loose_bound}.  We do so by computing an approximation to the error in a finite-dimensional subspace $Z^h \subset X$ (we do not exclude the possibility that $Z^h = X^h$ or $X^h \cap Z^h = \{0\}$), and denote this approximation by $\hat{e}_{\vmu}^h$.

We introduce the \textit{auxiliary} or \textit{error residual}
\begin{equation}\label{aux_res_definition}
	\rho_{\vmu}^h \coloneqq r_{\vmu}^h - \mathcal{L}_{\vmu}\hat{e}_{\vmu}^h.
\end{equation}
Analogous to the previous bound~\eqref{loose_bound}, with this form we arrive at
\begin{align}\label{auxiliary_bound}
\begin{split}
  \alpha(\vmu)\|e_{\vmu}^h - \hat{e}_{\vmu}^h\|_X^2 & \leq \left\|\mathcal{L}_{\vmu}\left(e_{\vmu}^h - \hat{e}_{\vmu}^h\right)\right\|_Y^2 = \|r_{\vmu}^h - \mathcal{L}_{\vmu}\hat{e}_{\vmu}^h\|_Y^2 = \|\rho_{\vmu}^h\|_Y^2,\\
  \Rightarrow \|e_{\vmu}^h - \hat{e}_{\vmu}^h\|_X & \leq \frac{\|\rho_{\vmu}^h\|_Y}{\sqrt{\alpha(\vmu)}}.
\end{split}
\end{align}
In the case that the approximation to the error is simply taken to be $\hat{e}^h_{\vmu} = 0$, then $\|\rho_{\vmu}^h\|_Y = \|r_{\vmu}^h\|_Y$.  However, if a reasonable approximation to the error is computed, it is often the case that $\|\rho_{\vmu}^h\|_Y \ll \|r_{\vmu}^h\|_Y$, resulting is less sensitivity to a small coercivity constant.

We use~\eqref{auxiliary_bound} and the triangle inequality to develop an alternative upper bound for $\|e_{\vmu}^h\|_X$:
\begin{equation}\label{tight_upper_bound}
	\|e_{\vmu}^h\|_X \leq \|\hat{e}_{\vmu}^h\|_X + \|e_{\vmu}^h - \hat{e}_{\vmu}^h\|_X \leq \|\hat{e}_{\vmu}^h\|_X + \frac{\|\rho_{\vmu}^h\|_Y}{\sqrt{\alpha(\vmu)}} \eqqcolon M^h(\vmu).
\end{equation}
With this form of the error bound we monitor its effectiveness with the so-called \textit{effectivity} ratio, defined as
\begin{equation}
  \frac{M^h(\vmu)}{\|e_{\vmu}^h\|_X},
\end{equation}
which we seek as close to one as possible.
The effectivity ratio is bounded in the following, which is adapted from~\cite{schmidt2020rigorous}:
\begin{theorem}\label{prop1}
  Fix $\delta \in \left[0,1\right)$ and $\vmu \in \mathcal{D}$.
  Let $u_{\vmu}$ be the solution to~\eqref{strong_form}, and let $u_{\vmu}^h$ be its discrete approximation in $X^h$, with residual $r_{\vmu}^h$ as defined in~\eqref{residual_definition}.  Denote the error as $e_{\vmu}^h$ (cf.~\eqref{error_definition}) and consider $\hat{e}_{\vmu}^h$ to be any approximation of this error.  Finally, let $\rho_{\vmu}^h$ denote the auxiliary residual (cf.~\eqref{aux_res_definition}).  If
\begin{equation}\label{stopping_criteria}
	\frac{\|\rho_{\vmu}^h\|_Y}{\sqrt{\alpha(\vmu)}\,\|\hat{e}_{\vmu}^h\|_X} \leq \delta,
\end{equation}
then the effectivity satisfies the following bound:
\begin{equation}
	\frac{M^h(\vmu)}{\|e_{\vmu}^h\|_X} \leq \frac{1 + \delta}{1 - \delta}
\end{equation}
\end{theorem}
\begin{proof}
Assume~\eqref{stopping_criteria} holds. By the reverse triangle inequality,~\eqref{auxiliary_bound}, and~\eqref{stopping_criteria}, we have
\begin{equation}\label{first_bound}
	\left|\frac{\|e_{\vmu}^h\|_X - \|\hat{e}_{\vmu}^h\|_X}{\|\hat{e}_{\vmu}^h\|_X}\right| \leq \frac{\|\hat{e}_{\vmu}^h - e_{\vmu}^h\|_X}{\|\hat{e}_{\vmu}^h\|_X} \leq \frac{\|\rho_{\vmu}^h\|_Y}{\sqrt{\alpha(\vmu)}\,\|\hat{e}_{\vmu}^h\|_X} \leq \delta.
\end{equation}

If $\|\hat{e}_{\vmu}^h\|_X > \|e_{\vmu}^h\|_X$, then it follows from~\eqref{first_bound} that
\begin{equation}
	\|\hat{e}_{\vmu}^h\|_X - \|e_{\vmu}^h\|_X \leq \delta\|\hat{e}_{\vmu}^h\|_X \implies (1-\delta)\|\hat{e}_{\vmu}^h\|_X \leq \|e_{\vmu}^h\|_X.
\end{equation}
If $\|\hat{e}_{\vmu}^h\|_X \leq \|e_{\vmu}^h\|_X$, then $(1-\delta)\|\hat{e}_{\vmu}^h\|_X \leq \|e_{\vmu}^h\|_X$ follows immediately since $\delta \geq 0$.  In either case,
\begin{equation}\label{e_hat_bound}
	(1-\delta)\|\hat{e}_{\vmu}^h\|_X \leq \|e_{\vmu}^h\|_X,
\end{equation}
holds.
Using~\eqref{e_hat_bound},~\eqref{auxiliary_bound}, and~\eqref{stopping_criteria}, it follows that
\begin{equation}\label{delta_inequality}
	\frac{\|\hat{e}_{\vmu}^h - e_{\vmu}^h\|_X}{\|e_{\vmu}^h\|_X} \leq \frac{\|\hat{e}_{\vmu}^h - e_{\vmu}^h\|_X}{(1-\delta)\|\hat{e}_{\vmu}^h\|_X} \leq \frac{\|\rho_{\vmu}^h\|_Y}{\sqrt{\alpha(\vmu)}(1-\delta)\|\hat{e}_{\vmu}^h\|_X} \leq \frac{\delta}{1 - \delta}.
\end{equation}
Finally, using the triangle inequality,~\eqref{stopping_criteria},~\eqref{e_hat_bound}, and~\eqref{delta_inequality}, we have
\begin{align}
\begin{split}
	M^h(\vmu) = \|\hat{e}_{\vmu}^h\|_X + \frac{\|\rho_{\vmu}^h\|_Y}{\sqrt{\alpha(\vmu)}} &\leq \|e_{\vmu}^h\|_X + \|\hat{e}_{\vmu}^h - e_{\vmu}^h\|_X + \frac{\|\rho_{\vmu}^h\|_Y}{\sqrt{\alpha(\vmu)}}\\
	 &\leq \|e_{\vmu}^h\|_X + \|\hat{e}_{\vmu}^h - e_{\vmu}^h\|_X + \delta \|\hat{e}_{\vmu}^h\|\\
	 &= \left( 1 + \frac{\|\hat{e}_{\vmu}^h - e_{\vmu}^h\|_X}{\|e_{\vmu}^h\|_X} + \delta\frac{\|\hat{e}_{\vmu}^h\|_X}{\|e_{\vmu}^h\|_X}\right)\|e_{\vmu}^h\|_X\\
	 &\leq \left(1 + \frac{\delta}{1 - \delta} + \frac{\delta}{1 - \delta}\right) \|e_{\vmu}^h\|_X\\
	 &= \left(\frac{1 + \delta}{1 - \delta}\right)\|e_{\vmu}^h\|_X,
	\end{split}
\end{align}
which completes the proof.
\end{proof}

\subsection{Application to Poisson's equation: Galerkin setting}\label{sec:appliedtopoisson}

Bound~\eqref{tight_upper_bound} is only useful if the inner products associated with the spaces $X$ and $Y$ are computable, and if there are easily constructed conforming subspaces $X^h,\,Z^h\subset X$.  We will demonstrate possible consequences by considering an example of the parameter-independent Poisson's equation with homogeneous boundary conditions:
\begin{align}\label{Poisson}
\begin{split}
	-\Delta u &= f, \hspace{7mm} \vx \in \Omega,\\
	u &= 0, \hspace{7mm} \vx \in \partial\Omega,
	\end{split}
\end{align}
with $\mathcal{L} \coloneqq -\Delta$.  From here, we have a number of different choices for the domain and range. One option is $X = H^2(\Omega)\cap H^1_0(\Omega)$, and $Y = L^2(\Omega)$; both norms are easily computable.  However, to compute discrete approximations $u^h$ and $\hat{e}^h$, we must construct finite element spaces $X^h$ and $Z^h$ that contain functions that are class $C^1(\Omega)$ across element boundaries, which are difficult to construct~\cite{boffi2013mixed}, and are not usually used in the numerical approximation of Poisson's equation.

Alternatively, consider $X = H_0^1(\Omega)$ and $Y = H^{-1}(\Omega) = \left(H_0^1(\Omega)\right)'$.  In this setting, the Poisson equation is often solved using variational methods, resulting in the Galerkin weak form of the equation:
\begin{equation}
	a(u,v) \coloneqq \int_\Omega \nabla u\cdot\nabla v\ dx = \int_\Omega fv\ dx \eqqcolon F(v), \hspace{6mm} \forall v \in H^1_0(\Omega).
\end{equation}
Using the language of duality pairings (see for example~\cite{braess2007finite}), it is possible to express this through a mapping $\mathcal{L}:X \to Y$, with $X = H^1_0(\Omega)$ and $Y = H^{-1}(\Omega)$, via
\begin{equation}\label{duality_operator}
	\mathcal{L}u\left[\cdot\right] \coloneqq \int_{\Omega}\nabla u \cdot \nabla\left[\cdot\right] \dif x.
\end{equation}

That is, $\mathcal{L}u = F \in H^{-1}(\Omega)$.  Standard conforming finite element spaces are readily constructed for $X$, and the norm on $X$ is easily computable.  However, the $Y = H^{-1}(\Omega)$ norm requires inversion of the Laplacian operator~\cite{bochev2009least}:
\begin{equation}
	 \| F\|_Y = \left((-\Delta)^{-1/2}F,(-\Delta)^{-1/2}F\right)_{0}^{1/2}.
\end{equation}
Consequently, to compute the $Y$-norm of the auxiliary residual $\rho_{\vmu}^h = f_{\vmu} + \Delta u_{\vmu}^h + \Delta \hat{e}_{\vmu}^h$, we would need to compute $(-\Delta)^{-1}f$, which is exactly the equation for which we seek an error bound.

In the finite element setting, the infinite dimensional space $H_0^1(\Omega)$ is not dealt with directly, but instead a finite dimensional \textit{test} subspace $V^h \subset H_0^1(\Omega)$ is introduced. The restriction of $F$ to the subspace $V^h$ is a bounded linear functional on $V^h$, so that $F$ is identified with an element in $\left(V^h\right)'$.  Thus, $\mathcal{L}u = F \in \left(V^h\right)'$, allowing us to associate $Y$ with $\left(V^h\right)'$.  While the norm for $Y=\left(V^h\right)'$ is more complex than either the $L^2$ or $H^1$ norms, it is still computable due to its finite dimension~\cite{rozza2007reduced}.

Unfortunately, the operator fails to be coercive in this case, which is seen either by using the fact that $X$ is infinite dimensional and $Y$ is finite dimensional, or by observing the standard Galerkin orthogonality condition:
\begin{equation}
	a(u-u^h,v^h) = 0,\hspace{6mm} \forall v^h \in V^h.
\end{equation}
This implies that
\begin{equation}
	\mathcal{L}(u - u^h) = 0 \in \left(V^h\right)'.
\end{equation}
That is, even if $u - u^h \neq 0$, the image $\mathcal{L}(u - u^h)$ is zero when considered as an element of the space $\left(V^h\right)'$.

Defining a finite-dimensional \textit{trial} subspace $W^h \subset H_0^1(\Omega)$ (where $W^h$ coincides with $V^h$ in the standard Galerkin method), standard ellipticity results~\cite{braess2007finite} show that $\mathcal{L}$ in~\eqref{duality_operator} is coercive on $X = W^h$.
However, the exact solution to~\eqref{strong_form} does not belong to $W^h$ in general.  For this reason, reduced basis approximations for standard Galerkin methods typically consider the ``true'' solution as a discrete solution in the finite dimensional subspace $W^h \subset X$.
As a result, it is not possible to apply the error bounds to an exact solution $u \notin W^h$ in the Ritz-Galerkin finite element setting.
In the next section, we show that this problem does not arise in a LSFEM context.

\section{The Least-Squares Finite Element Method}\label{sec:ls}

The least-squares finite element method (LSFEM) reformulates the PDE as a system of first-order equations and then defines the solution as the minimizer of a functional in an appropriate norm. See~\cite{bochev2009least,jiang2013least} for a complete description; a brief overview, with application to parametrized equations, is presented in this section.

\subsection{Abstract Formulation}

In addition to the assumptions of the previous section, we consider $\mathcal{L}_{\vmu}$ to be a bounded, linear first-order differential operator.
We wish to solve~\eqref{strong_form}.  Under the assumptions given by~\eqref{energy_balance}, any solution to~\eqref{strong_form} is the unique minimizer of the following problem:
\begin{equation}\label{CLSP}
	\argmin_{v\in X} J_{\vmu}(v;f_{\vmu}) \coloneqq \|\mathcal{L}_{\vmu}v - f_{\vmu}\|_Y^2.
\end{equation}
Conversely,~\eqref{CLSP} is guaranteed to have a unique minimizer $u_{\vmu} \in X$, and if $f_{\vmu}$ belongs to the range of $\mathcal{L}_{\vmu}$, this minimizer also solves~\eqref{strong_form}.  $u_{\vmu}$ necessarily satisfies the first-order optimality condition:
\begin{equation}\label{weak_form}
	(\mathcal{L}_{\vmu}u_{\vmu},\mathcal{L}_{\vmu}v)_Y = (f_{\vmu},\mathcal{L}_{\vmu}v)_Y, \hspace{6mm} \forall v \in X.
\end{equation}
For the remainder of the paper, we denote $u_{\vmu}$ as the unique solution to~\eqref{CLSP} and~\eqref{weak_form}; i.e., $\mathcal{L}_{\vmu}$ is surjective.

A LSFEM is defined by choosing a finite element subspace $X^h\subset X$, and seeking the minimum to~\eqref{CLSP} over this subspace instead.  The first-order optimality condition is now: find $u_{\vmu}^h \in X^h$ such that
\begin{equation}\label{discrete_weak}
	(\mathcal{L}_{\vmu}u_{\vmu}^h,\mathcal{L}_{\vmu}v^h)_Y = (f_{\vmu},\mathcal{L}_{\vmu}v^h)_Y, \hspace{6mm} \forall v^h \in X^h.
\end{equation}
Since $X^h\subset X$, coercivity of the bilinear form $a(\cdot,\cdot;\vmu) \coloneqq (\mathcal{L}_{\vmu}u_{\vmu}^h,\mathcal{L}_{\vmu}v^h)_Y$ and continuity of $F(\cdot) \coloneqq (f_{\vmu},\mathcal{L}_{\vmu}v^h)_Y$ on $X^h$ follow immediately.  Thus,~\eqref{discrete_weak} admits a unique solution $u_{\vmu}^h \in X^h$.\par 
The corresponding linear system of algebraic equations
\begin{equation}\label{eq:linear_system}
	A^h_{\vmu}\vu^h_{\vmu} = \vb_{\vmu}^h
\end{equation}
that is solved for the unknown vector of degrees of freedom $\vu^h_{\vmu}$, is also symmetric positive-definite.  This follows from the symmetry and coercivity of $a(\cdot,\cdot;\vmu)$.

\subsection{Error Approximation}

To approximate the error $e_{\vmu}^h = u_{\vmu} - u_{\vmu}^h$ for the LSFEM,  first an approximation $u_{\vmu}^h \in X^h \subset X$ is computed via~\eqref{discrete_weak}, and then the residual $r_{\vmu}^h = f_{\vmu} - \mathcal{L}_{\vmu}u^h_{\vmu}$ is computed.  Because of the form of~\eqref{discrete_weak}, the corresponding error satisfies
\begin{equation}
	(\mathcal{L}_{\vmu}e_{\vmu}^h,\mathcal{L}_{\vmu}v^h)_Y = 0, \hspace{6mm} \forall v^h \in X^h.
\end{equation}
As a result, if we attempt to compute an approximate error $\hat{e}_{\vmu}^h \in X^h$, we obtain zero.  To alleviate this, we introduce an additional subspace $Z^h$ that satisfies $X^h\subset Z^h \subset X$.  In the context of a finite element method, $Z^h$ may represent a refinement of the mesh, an increase in the polynomial order of the elements, or both.  We then solve for an approximation $Z^h \ni \hat{e}_{\vmu}^h \approx e_{\vmu}^h$ through the variational problem:
\begin{equation}
	(\mathcal{L}_{\vmu}\hat{e}_{\vmu}^h,\mathcal{L}_{\vmu}w^h)_Y = (r_{\vmu}^h,\mathcal{L}_{\vmu}w^h)_Y, \hspace{6mm} \forall w^h \in Z^h.
\end{equation}
Given the refinement of the space with $X^h \subset Z^h$, the auxiliary residual is expected to satisfy $\|\rho_{\vmu}^h\|_Y \ll \|r_{\vmu}^h\|_Y$. Thus, the rigorous error bound $M^h(\vmu)$ is applicable, and if the hypotheses of Theorem~\ref{prop1} hold, then the bounds on the effectivity are computable as well.\par 

It is the first-order formulation of the PDE that allows us to extend the theory  from \S~\ref{sec:parameqs} to form a practical method.  Any first-order formulation that leads to a practical LSFEM will lead to a space $X$ that is approximated by easily constructed finite element spaces, and a space $Y$ with an easily computable inner product. This leads to practical and computable norms $\|  \cdot \|_X$ and $\| \cdot \|_Y$.  We demonstrate this by continuing the example of the Poisson equation from \S~\ref{sec:appliedtopoisson}.  An equivalent first-order system of PDEs is given by
 \begin{align}\label{eq:ls_poisson}
\begin{split}
	\vq + \nabla u &= 0, \hspace{7mm} \vx \in \Omega,\\
	\nabla\cdot \vq &= 0, \hspace{7mm} \vx \in \Omega,\\
	u &= 0, \hspace{7mm} \vx \in \partial\Omega.
	\end{split}
\end{align}
The corresponding first-order differential operator is
\begin{align}\label{eq:poisson_first_order}
	\begin{split}
	    &\mathcal{L}: H(\text{div})\times H^1(\Omega) \to \left[L^2(\Omega)\right]^d \times L^2(\Omega),\\
		&\mathcal{L}\left[(\vq,u)\right] \coloneqq \begin{pmatrix}\vq + \nabla u\\
		\nabla\cdot \vq
 \end{pmatrix}.	
	\end{split}
\end{align}
The operator $\mathcal{L}$ satisfies \eqref{energy_balance} on the Hilbert spaces $X = H(\text{div})\times H^1$, $Y = \left(L^2\right)^d \times L^2$ \cite{pehlivanov1994least}, so the theory of \S~\ref{sec:parameqs} applies.  The norm of $Y = \left(L^2\right)^d \times L^2$ is easily computable;  moreover, simple conforming finite element spaces exist for $X = H(\text{div})\times H^1$~\cite{boffi2013mixed,raviart1977primal}.  Thus, the computational difficulties associated with the pairings $X = H_0^1$, $Y = H^{-1}$ and $X = H^2\cap H_0^1$, $Y=L^2$ from \S~\ref{sec:appliedtopoisson} are not present.
 
Furthermore, the LSFEM method minimizes the norm of the auxiliary residual $\|\rho_{\vmu}^h\|_Y$ by design.  This is a desirable property in light of the discussion in \S~\ref{sec:errorbounds}.

\section{Reduced Basis Methods}\label{sec:rbm}
In this section, we provide a brief overview of reduced basis (RB) methods for elliptic equations.  See~\cite{rozza2007reduced} for an extensive overview.
As described in \eqref{eq:linear_system}, the LSFEM discretization of a linear elliptic PDE leads to a parameter-dependent system of linear equations.  A Galerkin finite element method will also lead to a system of linear equations of the same form.  Thus, the algebraic considerations of reduced basis methods carry over from Galerkin methods to LSFEMs in a straightforward way.\par 

While we restrict our attention to steady-state problems, LSFEMs have also been successfully applied to time-dependent parabolic problems \cite{bochev2009least, tang1993least, yang1999some}.  Thus, for these two classes of problems, standard projection-based reduced order modeling approaches (e.g., Galerkin and Petrov-Galerkin \cite{bui2008model, carlberg2017galerkin, carlberg2011efficient}) can be applied.  Hyperbolic problems have proved to be more challenging for LSFEMs, although work has been done in this area \cite{bochev2001improved, de2004least,2003_Olson_thesis}.

\subsection{Galerkin Projection}

A parametrized elliptic PDE solved by a Galerkin variational method (e.g., a finite element method), leads to the equation:
\begin{equation}\label{bilinear_form}
	a(u^h_{\vmu}, v^h;\vmu) = F(v^h;\vmu), \hspace{6mm} \forall v^h \in X^h.
\end{equation}
Here, $a(\cdot,\cdot;\vmu):X \times X \to \mathbb{R}$ is a continuous and coercive bilinear form for all $\vmu \in \mathcal{D}\subset \mathbb{R}^d$, and $F(\cdot;\vmu):X \to \mathbb{R}$ is a bounded linear functional for all $\vmu$.

Let $N^h \coloneqq \text{dim}(X^h)$ and consider $\{\eta_j\}_{j=1}^{N^h}$ to be a basis for $X^h$. For any $\vmu$, the discrete solution has a representation $u^h_{\vmu} = \sum_{j=1}^{N^h}u_j(\vmu)\eta_j$, where $u_j(\vmu)$ denotes the coefficient to basis function $\eta_j$ and depending on $\vmu$.  Substitution into~\eqref{bilinear_form}, results in a linear system of the form
\begin{equation}\label{bilinear_form_algebraic}
	\sum_{j=1}^{N^h}a(\eta_j,\eta_i;\vmu)u_j(\vmu) = F(\eta_i;\vmu) \hspace{6mm} i=1,\dots,N^h.
\end{equation}
In a many query or real-time context,~\eqref{bilinear_form_algebraic} must be solved repeatedly or very quickly.  Even more, a large linear system must be assembled for \textit{each} parameter instance, which is prohibitively expensive for discretizations with many degrees of freedom.  Reduced basis methods are intended to help alleviate this cost.  By introducing a subspace $X^N\subset X$ with dimension $N \ll N^h$ and basis $\{\xi_j\}_{j=1}^N$, a reduced solution $u^N_{\vmu} = \sum_{j=1}^{N}c_j(\vmu)\xi_j$ is sought instead.  This leads to the much smaller linear system
\begin{equation}\label{bilinear_form_reduced}
	\sum_{j=1}^{N}a(\xi_j,\xi_i;\vmu)c_j(\vmu) = F(\xi_i;\vmu) \hspace{6mm} i=1,\dots,N.
\end{equation}

There are a number of features that distinguish a RB method.  First, an RB method must specify \textit{how} the reduced basis $\{\xi_j\}$ is constructed, which is part of the ``offline'' stage.
This ``offline-online'' decomposition is found throughout the reduced basis literature~\cite{quarteroni2015reduced, rozza2007reduced}.  Typically, the basis functions are linear combinations of the high-fidelity basis functions $\eta_j$.  We review the greedy sampling strategy for constructing the reduced basis in \S~\ref{sec:greedy}.

Second, a RB method requires the construction of an efficient error indicator $\widetilde{M}^N(\vmu)$ that quantifies the quality of the RB solution $u^N_{\vmu}$ in some manner.  This is used both to assess the quality of the computed RB solution in the online stage, and to guide the construction of the reduced basis when using a greedy sampling strategy in the offline stage. We review the standard error indicator used in reduced basis literature in \S~\ref{sec:rb_standard_err_ind} and discuss our improved error indicator for the least-squares  reduced basis method in \S~\ref{sec:lsrbm}.

Finally, a RB method is distinguished by the handling of the resulting reduced system~\eqref{bilinear_form_reduced}, which still requires considerable cost in the assembly process, despite the reduction, because each new value of $\vmu$ requires a new linear system and right-hand side.  The cost of this assembly is, in general, dependent on the dimension $N^h$, which is unacceptable for the many-query or real-time context.  Either additional assumptions on $a(\cdot,\cdot;\vmu)$ and $F(\cdot;\vmu)$ must be made, or an algorithm to remove this $N^h$ dependency must be specified. This issue is addressed by considering affinely parametrized equations.

\subsection{Affinely Parametrized Equations}\label{sec:affine_param}

A critical feature of an effective RB method is that the assembly of~\eqref{bilinear_form_reduced} should be independent of the dimension of the full-order problem $N^h$ to be useful in a many-query or real-time context.
A certain class of variational problems exist where an $N^h$-independent assembly process is readily obtained.
A variational problem is said to be \textit{affinely parametrized} if it can be expressed in the form
\begin{align}
	\begin{split}
		a(u,v;\vmu) &= \sum_{k=1}^{Q_a}\theta_k^a(\vmu)a_k(u,v),\\
		F(v;\vmu) &= \sum_{k=1}^{Q_F}\theta_k^F(\vmu)F_k(v).
	\end{split}
\end{align}
Here, $\{\theta_k^a\}_{k=1}^{Q_a}$ and $\{\theta_k^F\}_{k=1}^{Q_F}$ are a set of $Q_a$ (respectively $Q_F$) scalar functions of $\vmu$, the $\{a_k(u,v)\}_{k=1}^{Q_a}$ are continuous, parameter-independent, bilinear forms, and the $\{F_k\}_{k=1}^{Q_F}$ are continuous, parameter-independent, linear functionals.  When this is satisfied, equation~\eqref{bilinear_form_reduced} takes the form:
\begin{equation}
	\sum_{j=1}^N\left(\sum_{k=1}^{Q_a}\theta_k^a(\vmu)a_k(\xi_j,\xi_i)\right)c_j(\vmu) = \sum_{\ell = 1}^{Q_F}\theta_\ell^F(\vmu)F_\ell(\xi_i)\hspace{6mm} i = 1,2,\dots,N.
\end{equation}
That is, the system matrix and right hand side are simply linear combinations of the matrices and vectors
\begin{align}\label{reduced_system}
	\begin{split}
		\left(A_k\right)_{ij}&\coloneqq a_k(\xi_j,\xi_i)\\
		\left(\vb_k\right)_i &\coloneqq F_k(\xi_i).
	\end{split}
\end{align}
These are assembled in the offline stage, leading to an online stage that is independent of the problem size $N_h$.
While there there are RB methods that do not satisfy this property~---~e.g.,\ using the empirical interpolation method~\cite{barrault2004empirical}~---~the work here is restricted to affinely parametrized problems as in a host of other works~\cite{deparis2009, dihlmann2015, grepl2005,HB2013,huynh2007reduced,maday2002,sen2008,yano2014}.

The requirement for affinely parametrized equations is no more restrictive for the least-squares method than it is for the Galerkin case.
An example is in the case of the time-harmonic Maxwell's equation for the calculation of the electric field, $\vE$~\cite{HB2013}.  Let $\vJ$ be a known source term, $\mu$ the permeability, $\sigma$ the conductivity, $\epsilon$ the permittivity, $\omega$ the frequency, and $\beta = i\omega\sigma - \omega^2\epsilon$, where $i = \sqrt{-1}$.  The vector of parameters is thus $\vmu = (\mu,\sigma,\epsilon,\omega)$.  Introducing a test function $\vv$, the variational equation becomes
\begin{align}
\frac{1}{\mu}(\nabla \times \vE, \nabla \times \vv)_0 + \beta(\vE,\vv)_0 = i\omega(\vJ,\vv)_0,\hspace{6mm} \forall\vv \in H(\text{curl}).
\end{align}
where $(\cdot, \cdot)_0 $ is the $L^2(\Omega)$ inner-product for vector  valued functions.
The resulting weak equation is affinely parametrized.

A least-squares discretization is be obtained by introducing the variable $\vq = \mu^{-1}\nabla \times \vE$.  Introducing test functions $\vr$ and $\vv$, one obtains the weak formulation
\begin{align}
\begin{split}
&\left[(\nabla \times \vq,\nabla \times \vr)_0 + (\vq,\vr)_0\right] + \beta\left[(\nabla\times\vq,\vv)_0 + (\vE,\nabla\times\vr)_0\right] + \beta^2(\vE,\vv)_0\\
&-\frac{1}{\mu}\left[(\vq,\nabla\times\vv)_0 + (\nabla\times\vE,\vr)_0\right] + \frac{1}{\mu^2}(\nabla\times\vE,\nabla\times\vv)_0\\
&= i\omega(\vJ,\nabla\times\vr)_0 + i\beta\omega(\vJ,\vv)_0.
\end{split}
\end{align}
We see that the least-squares discretization also leads to an affinely parametrized variational equation.
\subsection{Error Indicator}\label{sec:rb_standard_err_ind}

For a reduced basis of dimension $N$ and for every $\vmu$, there is a corresponding RB solution $u_{\vmu}^N$ and a corresponding weak residual $R^N(\cdot;\vmu) \in \left(X^h\right)'$ defined as
\begin{equation}\label{dual_residual}
	R^N(v^h;\vmu) \coloneqq F(v^h;\vmu) - a(u^N_{\vmu},v^h;\vmu),\hspace{7mm} \forall v^h \in X^h.
\end{equation}
Reduced basis methods typically construct error indicators of the form
\begin{equation}\label{standard_indicator}
	\widetilde{M}^N(\vmu) \coloneqq \frac{\|R^N(\cdot;\vmu)\|_{\left(X^h\right)'}}{\beta_{\vmu}^{\text{LB}}},
\end{equation}
where $\beta_{\vmu}^{\text{LB}}$ is a lower bound of a coercivity or stability constant,  which is computed via the Successive Constraint Method (SCM)~\cite{chen2008monotonic,chen2009improved, huynh2010natural, huynh2007reduced,rozza2007reduced, rozza2013reduced,sen2006natural}.
SCM constructs a linear program of complexity independent of the problem size in the offline stage, similar to the construction of the reduced basis itself.

The indicator in~\eqref{standard_indicator} is an analogous quantity to
\begin{equation}
	\frac{\|\rho_{\vmu}^h\|_Y}{\sqrt{\alpha(\vmu)}}.
\end{equation}
In~\cite{schmidt2020rigorous}, the error indicator was improved upon by introducing an auxiliary error residual as in \S~\ref{sec:parameqs}.  However, as shown in \S~\ref{sec:appliedtopoisson}, an indicator based on the residual in~\eqref{dual_residual}  cannot be applied to the error with respect to an arbitrary function in $H^1$.  We refer to~\cite{quarteroni2015reduced} for a detailed explanation on the construction of $R^N$ and its corresponding dual norm.

\subsection{Offline and Online Stages using a Greedy Sampling Strategy}\label{sec:greedy}

The task of the offline stage in the reduced basis method is to construct the actual basis $\{\xi_i\}$.  A finite subset $\mathcal{D}_{\text{train}} \subset \mathcal{D}$ is chosen to represent the space of possible parameter values.  A parameter vector $\vmu_1 \in \mathcal{D}_{\text{train}}$ is chosen arbitrarily.

Define $\tilde{\xi}_1 \in X^h$ to be the solution of
\begin{equation}
	a(\tilde{\xi_1},v^h;\vmu_1) = F(v^h;\vmu_1), \hspace{5mm} \forall v^h\in X^h.
\end{equation}
Then the first reduced basis function is
\begin{equation}
	\xi_1 = \frac{\tilde{\xi_1}}{\|\tilde{\xi_1}\|_X}.
\end{equation}

Suppose for $N \geq 1$, an orthonormal basis $\{\xi_1,\dots,\xi_N\}$ has been constructed corresponding to parameters $\vmu_1,\dots,\vmu_N \in \mathcal{D}_{\text{train}}$.  For each $\vmu \in \mathcal{D}_{\text{train}}\setminus \{\vmu_1,\dots,\vmu_N\}$, let $u^N_{\vmu}$ be the solution to the projected variational problem
\begin{equation}\label{projected_variational}
	a(u^N_{\vmu},\xi_i;\vmu) = F(\xi_i;\vmu),\hspace{6mm} i=1,\dots,N.
\end{equation}
Using the expression for $R^N \in (X^h)'$ given by~\eqref{dual_residual}, the next parameter value is chosen through
\begin{equation}
	\vmu_{N+1} = \argmax\limits_{\vmu \in \mathcal{D}_{\text{train}}\setminus \{\vmu_1,\dots,\vmu_N\}} \widetilde{M}^N(\vmu),
\end{equation}
where $\widetilde{M}^N(\vmu)$ is defined in~\eqref{standard_indicator}.
The next basis function $\xi_{N+1}$ is found after computing the full-order solution $u^h_{\vmu_{N+1}}$ to equation~\eqref{projected_variational}, and orthonormalizing against the existing basis functions in the appropriate inner product.

The algorithm terminates after either the dimension of the basis has reached an upper bound or $\widetilde{M}^N(\vmu)$ is smaller than a preset tolerance.  At this point, the matrices and vectors from~\eqref{reduced_system} are computed.

In the subsequent online stage, having constructed a basis $\{\xi_1,\dots,\xi_N\}$, a reduced-order solution is easily obtained by solving the $N\times N$ linear system corresponding to the projected variational problem.  The computational cost is thus independent of the dimension of $X^h$, an essential component of a computationally efficient online stage.

\section{A Least-Squares Finite Element Method with Reduced Basis}\label{sec:lsrbm}  We now develop a least-squares based reduced basis method, which we label LSFEM-RB\@.
First, recall the \textit{improved} error estimate
\begin{equation}\label{eq:err_est}
	\|e_{\vmu}^h\|_X \leq \|\hat{e}_{\vmu}^h\|_X + \frac{\|\rho_{\vmu}^h\|_Y}{\sqrt{\alpha(\vmu)}} = M^h(\vmu),
\end{equation}
which is a rigorous upper bound for the error; its effectivity is also bounded by Theorem~\ref{prop1}.

Next, we make use of two finite-dimensional finite element spaces, $X^h \subset Z^h \subset X$, to compute the numerical approximation to the PDE and to the error equation.  To this end, we define
\begin{align}
	\begin{split}
		a(u,v;\vmu) &\coloneqq \left(\mathcal{L}_{\vmu}u,\mathcal{L}_{\vmu}v\right),\\
		F(v;\vmu) &\coloneqq \left(f_{\vmu},\mathcal{L}_{\vmu}v\right),\\
		R(w,u;\vmu) &\coloneqq F(w;\vmu) - a(u,w;\vmu).
	\end{split}
\end{align}

Recall the assumption of affine parametric dependence from \S~\ref{sec:affine_param},
\begin{align}\label{eq:affine_forms}
\begin{split}
a(u,v;\vmu) &= \sum_{k=1}^{Q_a}\theta_k^a(\vmu)a_k(u,v),\\
F(v;\vmu) &= \sum_{k=1}^{Q_f}\theta_k^F(\vmu)F_k(v),
	\end{split}
\end{align}
which is key for a computationally efficient online stage.
\par 

Finally, using the error estimate \eqref{eq:err_est} requires knowledge of the coercivity constant $\alpha(\vmu)$; replacing this by a lower bound $0 < \alpha_{\text{LB}}(\vmu) \leq \alpha(\vmu)$ also results in a rigorous upper bound for the error.  Computationally, we make use of a lower bound computed via the Successive Constraint Method \cite{chen2008monotonic}.  While this method guarantees a rigorous lower bound for the coercivity constant of a finite-dimensional subspace, it is still possible that it returns a value that is larger than the true coercivity constant $\alpha(\vmu)$.  This is addressed in more detail in \S~\ref{sec:coercivity}.
\paragraph{Remark} For the solution of the approximate error $\hat{e}_{\vmu}^h$, we use a space $Z^h$ that contains the original finite element space $X^h$; this is obtained through refinement of the mesh or increasing the polynomial order.  
As an alternative it is tempting to build the error approximation on a subspace $Z^h \subset X^h$; however, this will lead to the approximation $\hat{e}_{\vmu}^h = 0$ because of Galerkin orthogonality.  It is still possible to compute the error on a space $Z^h \not\subset X^h$,  with $Z^h$ having fewer degrees of freedom than $X^h$, as long as it is possible to transfer the solution $u_{\vmu}^h\in X^h$ onto the new grid.  However, if the grid corresponding to $Z^h$ is too coarse, the auxiliary residual $\|\rho_{\vmu}^h\|_Y$ will be too large, making the error estimate less effective.

\subsection{Offline Algorithm}
With an initial $\vmu_1 \in \mathcal{D}_{\text{train}}$ we compute the solution $u_{\vmu_1}^h$ to the equation
\begin{equation}\label{LS_primal}
	a(u_{\vmu_1}^h,v^h;\vmu_1) = F(v^h;\vmu_1), \hspace{7mm} v^h \in X^h.
\end{equation}
followed by the error approximation via the equation
\begin{equation}\label{LS_error}
a(\hat{e}^h_{\vmu_1},w^h;\vmu_1) = R(w^h,\tilde{\xi}_1;\vmu_1) \hspace{7mm} w^h\in Z^h.
\end{equation}
We then normalize $u_{\vmu_1}^h$ and $\hat{e}^h_{\vmu_1}$ to have unit $X$-norm and obtain our first pair of basis functions $\xi_1 \in X^h$ and $\phi_1 \in Z^h$.

Assume then that we we have constructed two orthonormal bases $\{\xi_1,\dots,\xi_N\} \subset X^h$ and $\{\phi_1,\dots,\phi_N\}\subset Z^h$ corresponding to parameters $\{\vmu_1,\dots,\vmu_N\}$.
For each $\vmu \in \mathcal{D}_{\text{train}}$, we compute the solution to the projected problem
\begin{equation}
	a(u^N_{\vmu},\xi_i;\vmu) = F(\xi_i;\vmu),\hspace{6mm} i = 1,\dots,N,
\end{equation}
and the corresponding projected error from
\begin{equation}
	a(\hat{e}^N_{\vmu},\phi_i;\vmu) = R(\phi_i,u^N_{\vmu};\vmu),\hspace{6mm} i = 1,\dots N.
\end{equation}
Defining the reduced residual by $r^N_{\vmu} = f_{\vmu} - \mathcal{L}_{\vmu}u^N_{\vmu}$, and the corresponding reduced auxiliary residual $\rho^N_{\vmu} = r^N_{\vmu} - \mathcal{L}_{\vmu}\hat{e}^N_{\vmu}$, the next parameter value is then selected through
\begin{equation}\label{eq:selection_criterion}
	\vmu_{N+1} = \argmax\limits_{\vmu \in \mathcal{D}_{\text{train}}\setminus \{\vmu_1,\dots,\vmu_N\}} M^N(\vmu) \coloneqq \|\hat{e}^N_{\vmu}\|_X + \frac{\|\rho^N_{\vmu}\|_Y}{\sqrt{\alpha_{\text{LB}}(\vmu)}}.
\end{equation}
Here, we denote $M^N(\vmu)$ as error estimate $M^h(\vmu)$ when restricted to approximations in the $N$-dimensional subspaces $\text{span}\{\xi_1,\dots,\xi_N\}$ and $\text{span}\{\phi_1,\dots,\phi_N\}$.  The basis functions $\xi_{N+1}$ and $\phi_{N+1}$ are obtained from the full-order solutions $u^h_{\vmu_{N+1}}$ and $\hat{e}^h_{\vmu_{N+1}}$ by orthonormalizing against the existing basis functions in the $X$-inner product.  

The algorithm terminates whenever
\begin{equation}\label{eq:stopping_criteria}
\frac{\|\rho^N_{\vmu}\|_Y}{\sqrt{\alpha(\vmu)}\|\hat{e}^N_{\vmu}\|_X} \leq \delta, \hspace{6mm} \forall \vmu \in \mathcal{D}_{\text{train}}.
\end{equation}
for a prescribed tolerance $\delta \in (0,1)$.  During the course of the offline algorithm, $\delta$ is increased if a full-order error estimate is encountered that exceeds the current value; if the full-order error indicator for a given $\vmu$ value is not bounded by $\delta$, then we cannot expect a reduced-order analogue to be bounded by this quantity either.  If $\delta$ is too large at the end of the algorithm, mesh or polynomial refinement of the space is needed to increase accuracy.\par 
Once the basis functions have been constructed, the reduced basis matrices with entries $a_k(\xi_i,\xi_j)$, $a_k(\phi_i,\phi_j)$, $a_k(\phi_i,\xi_j)$ and reduced basis vectors with entries $F_k(\xi_i)$ and $F_k(\phi_i)$ are assembled.  The algorithm for the offline stage is given in Algorithm~\ref{ls-algorithm}.

\begin{algorithm}
\caption{Least Squares Reduced Basis Offline Algorithm}
\label{ls-algorithm}
\begin{algorithmic}
\State Choose $\vmu_1 \in \mathcal{D}_{\text{train}}$
\State Compute full-order solutions $u^h_{\vmu_1}$ and $\hat{e}^h_{\vmu_1}$.\Comment{\eqref{LS_primal}, \eqref{LS_error}}
\State Normalize to obtain primal basis $\{\xi_1\}$, and error basis $\{\phi_1\}$.
\For {$n=1,\dots,N_{\text{max}}$}
	\If {$\frac{\|\rho^n_{\vmu}\|_Y}{\sqrt{\alpha_{\text{LB}}(\mu)}\|\hat{e}_{\vmu}^n\|_X} \leq \delta$ for all $\vmu \in \mathcal{D}_{\text{train}}$}\Comment{\eqref{eq:stopping_criteria}}
		\State \textbf{Break}
	\EndIf
	\State $\vmu_{n+1} = \argmax M^n(\vmu)$\Comment{\eqref{eq:selection_criterion}}
	\State Compute full-order solutions $u^h_{\vmu_{n+1}}$ and $\hat{e}^h_{\vmu_{n+1}}$.
	\State If full-order estimate $\frac{\|\rho^h_{\vmu_{n+1}}\|_Y}{\sqrt{\alpha_{\text{LB}}(\vmu_{n+1})}\|\hat{e}_{\vmu_{n+1}}^h\|_X} > \delta$, set $\delta = \frac{\|\rho^h_{\vmu_{n+1}}\|_Y}{\sqrt{\alpha_{\text{LB}}(\vmu_{n+1})}\|\hat{e}_{\vmu_{n+1}}^h\|_X}$.
	\State Orthonormalize $u^h_{\vmu_{n+1}}$ against $\{\xi_1,\dots,\xi_n\}$ to obtain $\xi_{n+1}$, and append.
	\State Orthonormalize $\hat{e}^h_{\vmu_{n+1}}$ against $\{\phi_1,\dots,\phi_n\}$ to obtain $\phi_{n+1}$, and append.
\EndFor
\State Assemble matrices $a_k(\xi_i,\xi_j)$, $a_k(\phi_i,\phi_j)$, and $a_k(\phi_i,\xi_j)$.\Comment{\eqref{eq:affine_forms}}
\State Assemble vectors $F_k(\xi_i)$ and $F_k(\phi_i)$.\Comment{\eqref{eq:affine_forms}}
\end{algorithmic}
\end{algorithm}

If the algorithm terminates with $N < N_{\text{max}}$, and $\delta < 1$ then \eqref{eq:stopping_criteria} and Theorem~\ref{prop1} imply
\begin{equation}
	\frac{M^N(\vmu)}{\|u_{\vmu} - u^N_{\vmu}\|_X} \leq \frac{1 + \delta}{1 - \delta},\hspace{6mm} \forall \vmu \in \mathcal{D}_{\text{train}}.
\end{equation}

Thus, we obtain an upper bound for the effectivity ratio in $\mathcal{D}_{\text{train}}$, in addition to the error itself.

\subsection{Online Algorithm}
In the online stage, for a new parameter $\vmu$, the corresponding RB approximation and error bound is computed as follows.  First, the $N\times N$ projected problem is assembled via
\begin{align}\label{eq:online_primal}
\begin{split}
A_N(\vmu) &= \sum_{k=1}^{Q_a}\theta_k^a(\vmu)a_k(\xi_j,\xi_i),\\
\vb_N(\vmu) &= \sum_{m=1}^{Q_f}\theta_m^F(\vmu)F_m(\xi_i).
	\end{split}
\end{align}
Since the parameter-independent terms in \eqref{eq:online_primal} are assembled in the offline stage, assembly in the online stage incurs a cost of $\mathcal{O}(Q_aN^2 + Q_fN)$.  The resulting dense system
$A_N(\vmu) \vc_N = \vb_N(\vmu)$
 is solved directly, incurring a cost of $\mathcal{O}(N^3)$.
Here $\vc_N\in \mathbb{R}^N$ are the the RB coefficients, that is, $u_N(\vmu) = \sum_{n=1}^N [\vc_N]_n \xi_n$.\par 
To compute the approximation of the error, the $N\times N$ matrix 
\begin{align}\label{eq:error_matrix}
	\widehat{A}_N(\vmu) = \sum_{k=1}^{Q_a}\theta_k^a(\vmu)a_k(\phi_j,\phi_i)
\end{align}
is assembled, which is once again the linear combination of pre-assembled matrices; the cost is $\mathcal{O}(Q_a N^2)$.  To compute the right-hand side of the error equation, the following computations are performed:
\begin{align}\label{eq:error_rhs}
	\begin{split}
		G_k(\phi_i) &= \sum_{j=1}^Na_k(\xi_j,\phi_i)[\vc_N]_j\\
		\widehat{\vb}_N(\vmu) &= \sum_{m=1}^{Q_f}\theta_m^F(\vmu)F_m(\phi_i) - \sum_{k = 1}^{Q_a}\theta_k^a(\vmu)G_k(\phi_i).
	\end{split}
\end{align}
This requires $\mathcal{O}(Q_aN^2)$ operations to perform the matrix vector multiplications for $G_k$, and $\mathcal{O}\left((Q_a + Q_f)N\right)$ cost to form the necessary linear combination of vectors.  Finally, the solution of this system is also $\mathcal{O}(N^3)$, which determines the RB error coefficients $\widehat{\vc}_N \in \mathbb{R}^N$, that is, $\hat{e}^N = \sum_{n=1}^N [\hat{c}_N]_n \phi_n$.
\par 
To form the error bound, the quantities $\|\hat{e}^N_{\vmu}\|_X$ and $\|\rho^N_{\vmu}\|_Y$ must be computed.  Since the error basis functions $\{\phi_j\}_{j=1}^N$ were orthonormalized in the (computable) $X$-norm, the first quantity is easily computed by
\begin{align}\label{eq:error_norm}
	\|\hat{e}^N_{\vmu}\|_X = \sqrt{\widehat{\vc}_N^T\widehat{\vc}_N}.
\end{align}
To compute the norm of the auxiliary residual, we compute
\begin{align}\label{eq:aux_res_expand}
	\begin{split}
		\|\rho_{\vmu}^N\|_Y^2 &= \|f_{\vmu} - \mathcal{L}_{\vmu}u_{\vmu}^N- \mathcal{L}_{\vmu}\hat{e}_{\vmu}^N\|_Y^2\\
		&= (f_{\vmu},f_{\vmu})_Y + \left[(\mathcal{L}_{\vmu}u_{\vmu}^N,\mathcal{L}_{\vmu}u_{\vmu}^N)_Y - (f_{\vmu},\mathcal{L}_{\vmu}u_{\vmu}^N)_Y\right]\\
		&+\left[(\mathcal{L}_{\vmu}\hat{e}_{\vmu}^N,\mathcal{L}_{\vmu}\hat{e}_{\vmu}^N)_Y + (\mathcal{L}_{\vmu}u_{\vmu}^N,\mathcal{L}_{\vmu}\hat{e}_{\vmu}^N)_Y - (f_{\vmu},\mathcal{L}_{\vmu}\hat{e}_{\vmu}^N)_Y\right]\\
		&-(f_{\vmu},\mathcal{L}_{\vmu}u_{\vmu}^N)_Y - \left[(f_{\vmu},\mathcal{L}_{\vmu}\hat{e}_{\vmu}^N)_Y - (\mathcal{L}_{\vmu}u_{\vmu}^N,\mathcal{L}_{\vmu}\hat{e}_{\vmu}^N)_Y\right]\\
		&= (f_{\vmu},f_{\vmu})_Y - (f_{\vmu},\mathcal{L}_{\vmu}u_{\vmu}^N)_Y - \left[(f_{\vmu},\mathcal{L}_{\vmu}\hat{e}_{\vmu}^N)_Y - (\mathcal{L}_{\vmu}u_{\vmu}^N,\mathcal{L}_{\vmu}\hat{e}_{\vmu}^N)_Y\right]\\
		&= (f_{\vmu},f_{\vmu})_Y - \vb_N^T\vc_N - \widehat{\vb}^T\widehat{\vc}_N.
	\end{split}
\end{align}
Thus, the auxiliary residual is computed through inner products between vectors in $\mathbb{R}^N$ and the computation of $(f_{\vmu},f_{\vmu})_Y$.  For LSFEM, the $Y$-norm usually corresponds to the $L^2$-inner product, the affine parametric dependence of the problem ensures the quantity $(f_{\vmu},f_{\vmu})_Y$ is computed at a cost independent of the problem size, either analytically or through quadrature.\par 
If a method such as SCM is used to compute $\alpha_{\text{LB}}(\vmu)$ at a cost independent of problem size, every necessary computation of the online stage is performed at a cost that is independent of the size of the high-fidelity problem.  Finally, the error estimate is computed via
\begin{equation}\label{eq:final_error_est}
M^N(\vmu) = \|\hat{e}_{\vmu}^N\|_X + \frac{\|\rho_{\vmu}^N\|_Y}{\sqrt{\alpha_{\text{LB}}(\vmu)}}	
\end{equation}

These steps are collected in Algorithm \ref{alg:online}.

\begin{algorithm}
\caption{Least Squares Reduced Basis Online Algorithm}
\label{alg:online}
\begin{algorithmic}
\State Input: parameter $\vmu \in \mathbb{R}^P$
\State Assemble RB system: $A_N(\vmu)$ and  $\vb_N(\vmu)$\Comment{\eqref{eq:online_primal}}
\State Solve RB system $A_N(\vmu)\vc_N = \vb_N(\vmu)$
\State Assemble RB error system: $\widehat{A}_N(\vmu)$, $G_k(\phi_i)$ and $\widehat{\vb}_N(\vmu)$ \Comment{\eqref{eq:error_matrix},\eqref{eq:error_rhs}}
\State Solve RB error system $\widehat{A}_N(\vmu)\widehat{\vc}_N = \widehat{\vb}_N(\vmu)$
\State Compute error norm $\|\hat{e}^N_{\vmu}\|_X = \sqrt{\widehat{\vc}_N^T\widehat{\vc}_N}$ \Comment{\eqref{eq:error_norm}}
\State Compute auxiliary residual $\|\rho_{\vmu}^N\|_Y^2 = (f_{\vmu},f_{\vmu})_Y - \vb_N^T\vc_N - \widehat{\vb}^T\widehat{\vc}_N$\Comment{\eqref{eq:aux_res_expand}}
\State Compute $\alpha_{\text{LB}}(\vmu)$ (e.g. through online phase of SCM)
\State Compute error estimate $M^N(\vmu) = \|\hat{e}_{\vmu}^N\|_X + \frac{\|\rho_{\vmu}^N\|_Y}{\sqrt{\alpha_{\text{LB}}(\vmu)}}$ \Comment{\eqref{eq:final_error_est}}
\end{algorithmic}
\end{algorithm}

\subsection{The Coercivity Constant and the Rigor of the Error Estimate}\label{sec:coercivity}

  We make use of a lower bound $\alpha_{\text{LB}}$ in Algorithm~\ref{ls-algorithm} to the coercivity constant $\alpha$. In practice, this lower bound is approximated  via the SCM,
  which computes a lower bound to the discrete coercivity constant $\alpha^h(\vmu)$. We now examine the implication of this approximation.

The coercivity constant $\alpha(\vmu)$ is the infimum of a Rayleigh Quotient
\begin{align}
	\alpha(\vmu) = \inf_{u \in X}\frac{\|\mathcal{L}_{\vmu}u\|_Y^2}{\|u\|_X^2}.
\end{align}
The discrete coercivity constant $\alpha^h(\vmu)$ is the infimum of the same Rayleigh Quotient, but $u$ is restricted to be in the finite-dimensional approximation space:
\begin{align}
	\alpha^h(\vmu) = \inf_{u^h \in X^h}\frac{\|\mathcal{L}_{\vmu}u^h\|_Y^2}{\|u^h\|_X^2}.
\end{align}
In a conforming discretization, since $X^h \subset X$, it follows that $\alpha(\vmu)\leq \alpha^h(\vmu)$.  Thus there is a possibility that
\begin{align}
	\alpha(\vmu) < \alpha_{\text{LB}}(\vmu) \leq \alpha^h(\vmu).
\end{align}
This leads to a potential underestimation of the error in $M^h(\vmu)$ or $M^N(\vmu)$.  However, extensive research has been done with respect to the convergence of finite element approximations of eigenvalue problems; see \cite{boffi2010finite} for a thorough overview.\par
If the problem
\begin{align}\label{eq:eig_variational}
	(\mathcal{L}_{\vmu}u,\mathcal{L}_{\vmu}v)_Y = (f,v)_X, \hspace{4mm} \forall v \in X,
\end{align}
has a \emph{compact} solution operator $T_{\vmu}:X \to X$, where $T_{\vmu}f = u$ is the solution to the variational problem \eqref{eq:eig_variational}, then the error $\alpha^h(\vmu) - \alpha_(\vmu)$ is bounded by the \emph{square} of the approximation error of the finite dimensional space $X^h$.  Thus, if a LSFEM with order of convergence $r$ is used, we would expect
\begin{equation}
	\alpha^h(\vmu) \leq \alpha(\vmu) + \mathcal{O}(h^{2r}),
\end{equation}
where $h$ is the mesh size.  That is, the error in the approximation of the coercivity constant is much lower than the error in the LSFEM solution.
Thus, in the worst case, a non-rigorous lower bound $\alpha_{\text{LB}}$ computed by the successive constraint method would still be bounded by $\alpha(\vmu) + \mathcal{O}(h^{2r})$ and thus for sufficiently small $h$,
\begin{align}
	\begin{split}
		\frac{1}{\sqrt{\alpha(\vmu)}} \leq \frac{1}{\sqrt{\alpha_{\text{LB}}(\vmu) - \mathcal{O}(h^{2r})}} = \frac{1}{\sqrt{\alpha_{\text{LB}}(\vmu)}}(1 + \mathcal{O}(h^{2r})).
	\end{split}
\end{align}
If this estimate holds, the lower bound built by SCM is asymptotically rigorous, and so is the corresponding error estimate $M^N(\vmu)$.\par 
Showing the compactness of the solution operator of \eqref{eq:eig_variational} will be the focus of future research, but numerical experiments have shown evidence of the higher order convergence rate.  In Figure \ref{fig:coercivity}, convergence to the exact coercivity constant of an LSFEM applied to the ordinary differential equation $-u'' = f$ with homogeneous Dirichlet boundary conditions is shown. In this case, the first-order reformulation leads to the operator $\mathcal{L}: H^1(\Omega) \times H_0^1(\Omega) \to L^2(\Omega)\times L^2(\Omega)$ defined as
\begin{equation}\label{eq:first_order_op}
	\mathcal{L}\left[(q,u)\right] = \begin{pmatrix}
		q + u'\\
		q'
	\end{pmatrix},
\end{equation}
with coercivity constant $\alpha = 1 - (1 + \sqrt{1 + 4\pi^2})/2(1 + \pi^2)$.  Using piecewise linear finite elements with 1st order convergence, we see the expected 2nd order convergence of the discrete coercivity constant.
\begin{figure}[ht]
\centering
	\includegraphics[scale=0.45]{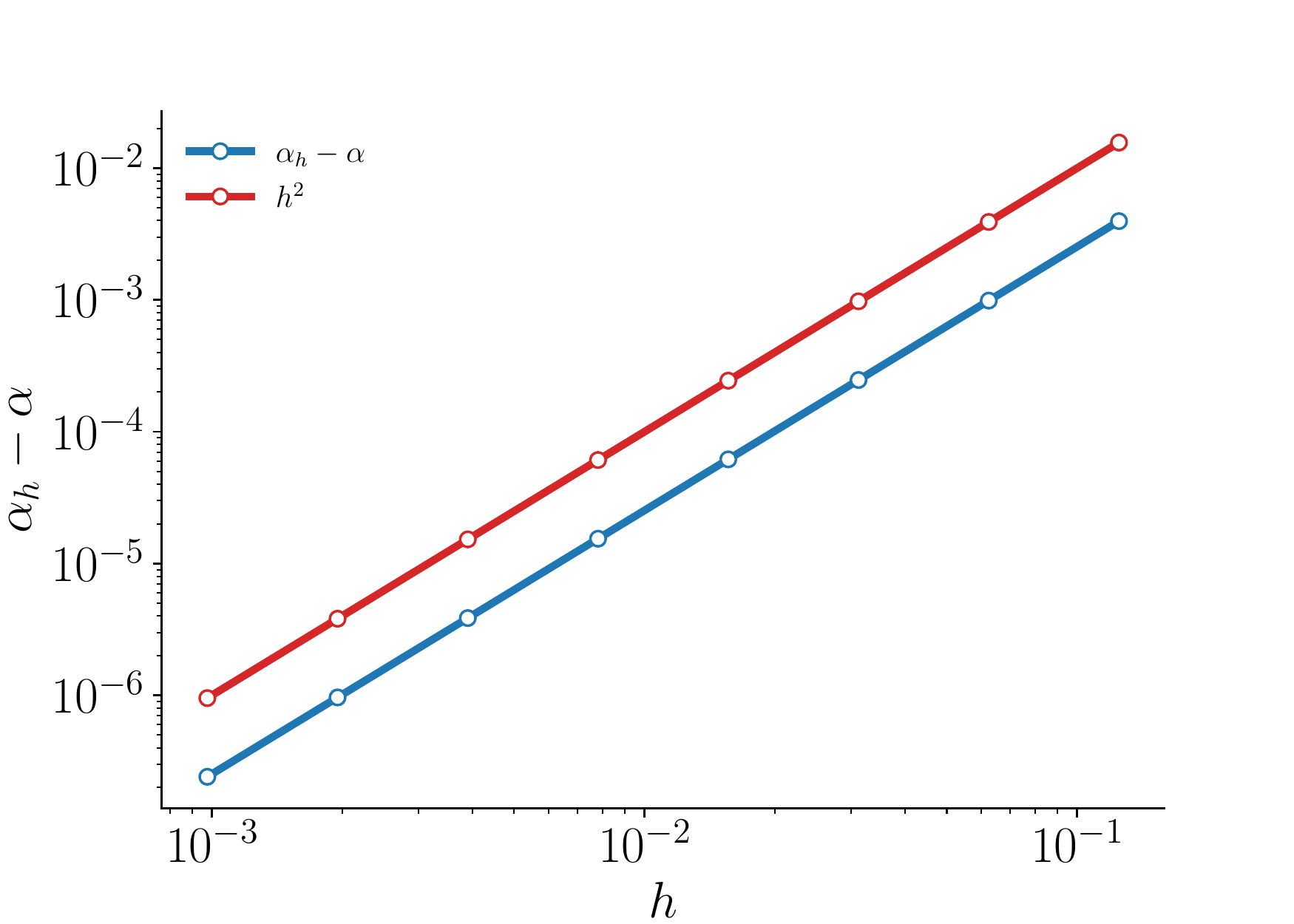}
	\caption{Convergence of the discrete coercivity constant for the operator defined in \eqref{eq:first_order_op}.  A piecewise linear finite element space with mesh spacing $h$ is used.}\label{fig:coercivity}
\end{figure}

\section{Numerical Evidence}\label{sec:numerics}

In this section we present numerical numerical evidence in support of the
LSFEM-RB method introduced in \S~\ref{sec:lsrbm}.  A single parameter
study is given in \S~\ref{results-thermal-1} in order to detail the bounds
on the error, while a three-parameter study is discussed in
\S~\ref{sec:thermal3}.  Finally, in \S~\ref{sec:elasticity}, the method
is applied to an elasticity problem to highlight robustness.  The software library
Firedrake~\cite{rathgeber_etal_firedrake} is used in the following tests. Moreover, it is easy to check that the numerical examples considered below  are affinely parametrized in the least-squares setting.

\subsection{Thermal Block~---~1 Parameter}\label{results-thermal-1}

We first apply the LSFEM-RB framework to a standard test problem in the reduced basis community, the ``thermal block'' problem~\cite{haasdonk2017reduced, rozza2007reduced, schmidt2020rigorous}.
The governing partial differential equation is a variable coefficient Poisson problem:
\begin{align}
\begin{split}
	-\nabla \cdot\kappa(\bm{x})\nabla u &= 0 \hspace{6mm} \text{in } \Omega,\\
	u &= 0 \hspace{6mm} \text{on } \Gamma_D,\\
	\kappa(\bm{x})\nabla u\cdot \bm{n} &= g, \hspace{6mm} \text{on } \partial\Omega\setminus\Gamma_D.
\end{split}
\end{align}
Here, $\Omega$ is the unit square, $\Gamma_D = \{(x,y) \in \partial\Omega\  |\  y = 1\}$, $g(x,y)$ is a function satisfying $g(0,y) = g(1,y) = 0$ and $g(x,0) = 1$, and $\kappa(\bm{x})$ is a piecewise constant function taking two different values in subdomains $\Omega_1,\Omega_2$; see Figure~\ref{one_parameter_domain}.  Specifically,
\begin{equation}
  \kappa(\vx) =
  \begin{cases}
    \mu & \vx \in \Omega_1\\
      1 & \vx \in \Omega_2,
  \end{cases}
\end{equation}
with $\mu \in [10^{-1},10^{1}]$.
\begin{figure}[ht]
\begin{center}
\begin{tikzpicture}
\draw (0,0) rectangle (4,4);
\draw (2,0) -- (2,4);
\node at (1,2) {$\kappa(\bm{x}) = \mu$};
\node at (1,1) {$\Omega_1$};
\node at (3,2) {$\kappa(\bm{x}) = 1$};
\node at (3,1) {$\Omega_2$};
\end{tikzpicture}
\caption{Variable Poisson problem with conductivity in two subdomains.}\label{one_parameter_domain}
\end{center}
\end{figure}
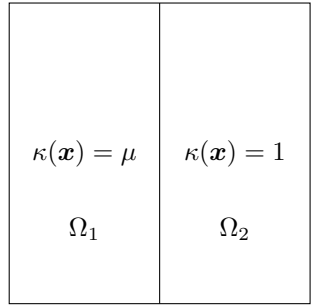

By introducing a constant lifting function $\bm{q}_{\ell} = (0,-1)^T$, and defining the flux variable
\begin{equation}
	\bm{q} = -\kappa\nabla u + \bm{q}_{\ell},
\end{equation}
the following equivalent first order system is obtained:
\begin{align}\label{operator-thermal-1}
\begin{split}
	\kappa^{-1/2}\bm{q} + \kappa^{1/2}\nabla u &= \kappa^{-1/2}\bm{q}_\ell \hspace{6mm} \text{in }\Omega,\\
	\nabla \cdot \bm{q} &= 0 \hspace{16mm} \text{in }\Omega,\\
	u &= 0 \hspace{16mm} \text{on }\Gamma_D,\\
	\bm{q}\cdot\bm{n} &= 0 \hspace{16mm} \text{on }\partial\Omega\setminus\Gamma_D.
\end{split}
\end{align}
The first order system~\eqref{operator-thermal-1} defines an operator $\mathcal{L}_{\vmu}$ with domain $X \subset H(\text{div})\times H^1(\Omega)$ and range $Y = (L^2(\Omega))^2 \times L^2(\Omega)$; here $X$ is the subspace of functions that satisfy the homogeneous boundary conditions.  It is shown in~\cite{bochev2009least, pehlivanov1994least} that the resulting operator $\mathcal{L}_{\vmu}$ satisfies~\eqref{energy_balance}~---~i.e., is continuous and has a bounded inverse~---~with respect to the $H(\text{div})\times H^1$ norm.  Using this norm on $X$ leads to a well-posed problem and the applicability of the error estimate~\eqref{tight_upper_bound}.

We compute an approximation using the subspace $X^h = (\text{RT}_0) \times P_1$, approximating $\bm{q}$ by the lowest order Raviart-Thomas space~\cite{raviart1977primal} and $u$ by piecewise linear polynomials.  The approximations $\bm{q}^h$ and $u^h$ are computed on a mesh corresponding to 1,016 degrees of freedom.

The reduced basis is constructed using a sample of 50 logarithmically spaced samples $\mu \in [0.1, 10.]$.  The auxiliary error equation is solved on the same mesh using $(\text{RT}_1)\times P_2$ elements, corresponding to 3,556 degrees of freedom.\par
The offline algorithm consists of two stages: the offline SCM portion and the construction of the reduced basis.  The SCM requires the solution of 11 generalized eigenvalue problems of size $1,016\times 1,016$.  For the construction of the reduced basis, the algorithm terminates after computing only $N = 3$ full-order solutions are required.  The final tolerance upon termination is approximately $\delta \approx 0.3984$.
Thus, the error estimate is guaranteed to satisfy the effectivity bound
\begin{align}\label{thermal-1-effectivity-bound}
	\frac{M^N(\mu)}{\|e_{\mu}^N\|_X} \leq \frac{1 + \delta}{1 - \delta} \approx 2.3244.
\end{align}
Thus, our error bound overestimates the true error by at worst a factor of approximately 2.3244.

To test the reduced basis, we generate 100 randomly sampled parameter values $\mu \in [0.1, 10.0]$ and compute a reference solution using $(\text{RT}_2)\times P_3$ elements after performing two uniform mesh refinements.  This corresponds to 121,920 degrees of freedom.  We then compute the reduced basis approximation for these parameter values and the corresponding RB error estimate.  For each parameter value, the lower bound for the coercivity constant is computed through the online SCM algorithm; the resulting linear program has 3 variables and 16 inequality constraints.  The reduced basis solution and approximate error requires the solution of two $3\times 3$ linear systems.  The true error and the corresponding error estimate are shown for the test set in Figure~\ref{thermal-1-error}.  The error bound is rigorous and resolves the difference in error throughout the parameter domain.

\begin{figure}[ht]
\centering
\begin{subfigure}[b]{.48\textwidth}
  \vspace{0pt}
	\centering
  \includegraphics[height=1.76in]{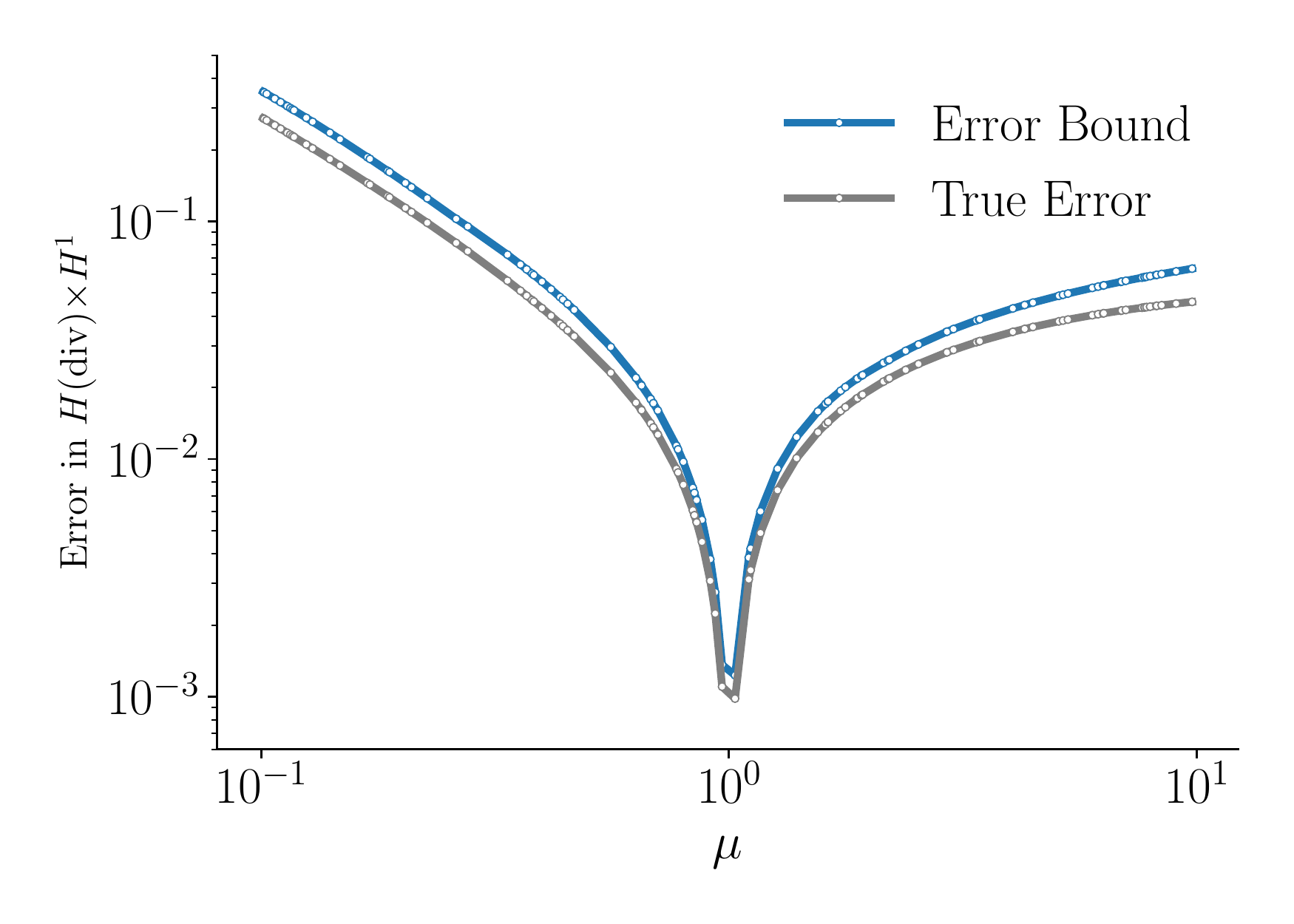}
  \caption{}\label{thermal-1-error}
\end{subfigure}\hfill
\begin{subfigure}[b]{.48\textwidth}
  \vspace{0pt}
	\centering
	\includegraphics[height=1.76in]{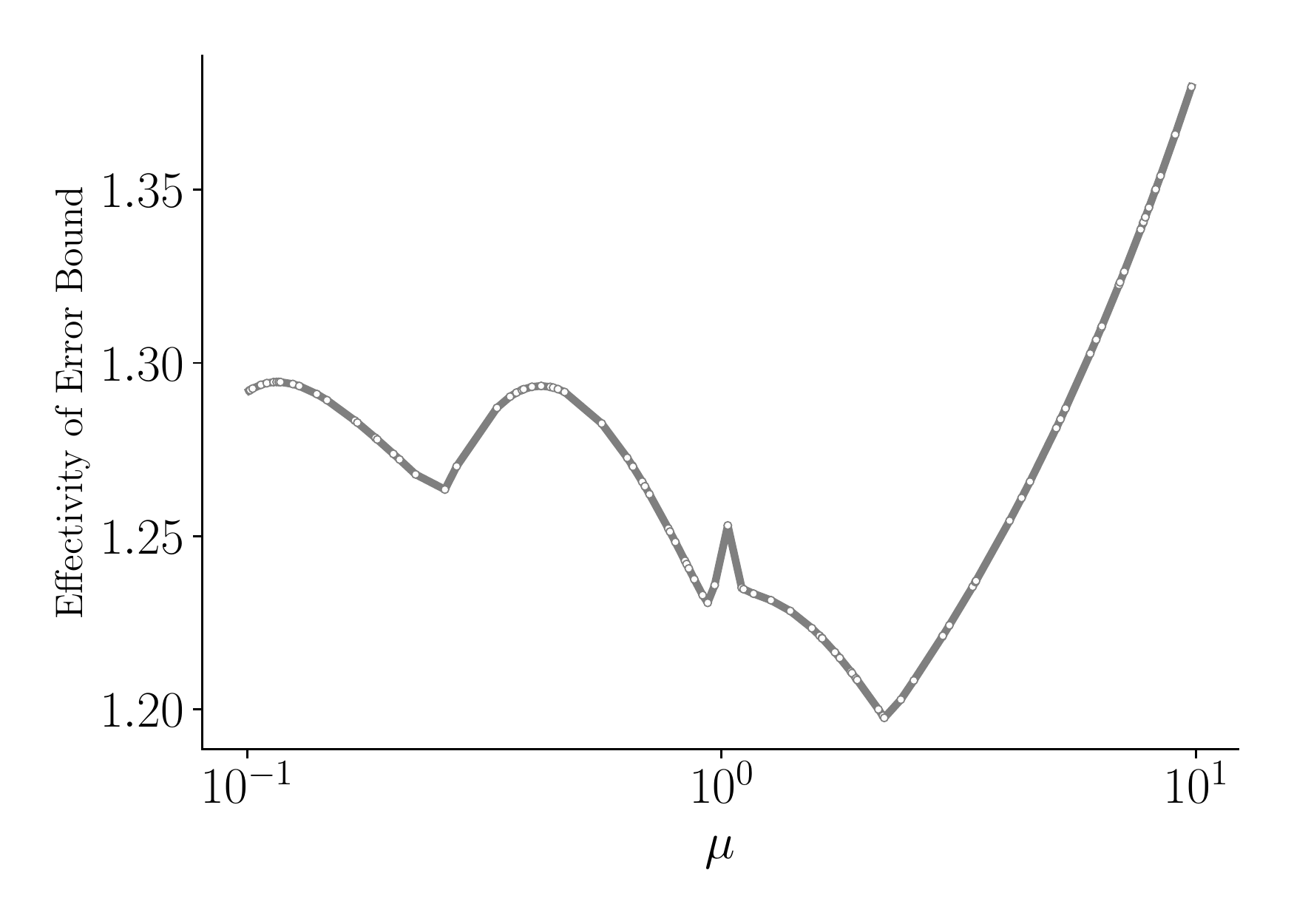}
	\caption{}\label{thermal-1-effectivity}
\end{subfigure}
\caption{(a) The error between the RB solution and the reference solution, along with the corresponding RB error bound over the testing set of parameter values of $\mu$.  (b) The effectivity of the RB error estimate in~\eqref{thermal-1-effectivity-bound} over the testing set of parameter values.}
\end{figure}

In Figure~\ref{thermal-1-effectivity}, we plot the effectivity of the error estimate over the same testing set of data.  The error bound overestimates the error by a small factor, less than $1.40$, which outperforms the effectivity bound in~\eqref{thermal-1-effectivity-bound}.\par

Finally, we plot the run-time for each new parameter value encountered in the online phase for both the reduced basis method, and solutions computed using only full-order solutions in Figure \ref{fig:thermal-1-run_time}.\footnote{The compute node uses two Intel Xeon E5-2630 processors.}  The reduced basis method incurs an offline cost of approximately 4 seconds; thus, for 16 parameter values or fewer, the reduced basis approach is more expensive.  However, because the size of the linear system is only $3\times 3$, total run-time grows extremely slowly for additional parameters.  Thus after 17 parameter values, the reduced basis approach becomes more computationally efficient, with rapidly increasing computational gains as the number of online parameters grows.

\begin{figure}[ht]
\centering
	\includegraphics[scale=0.50]{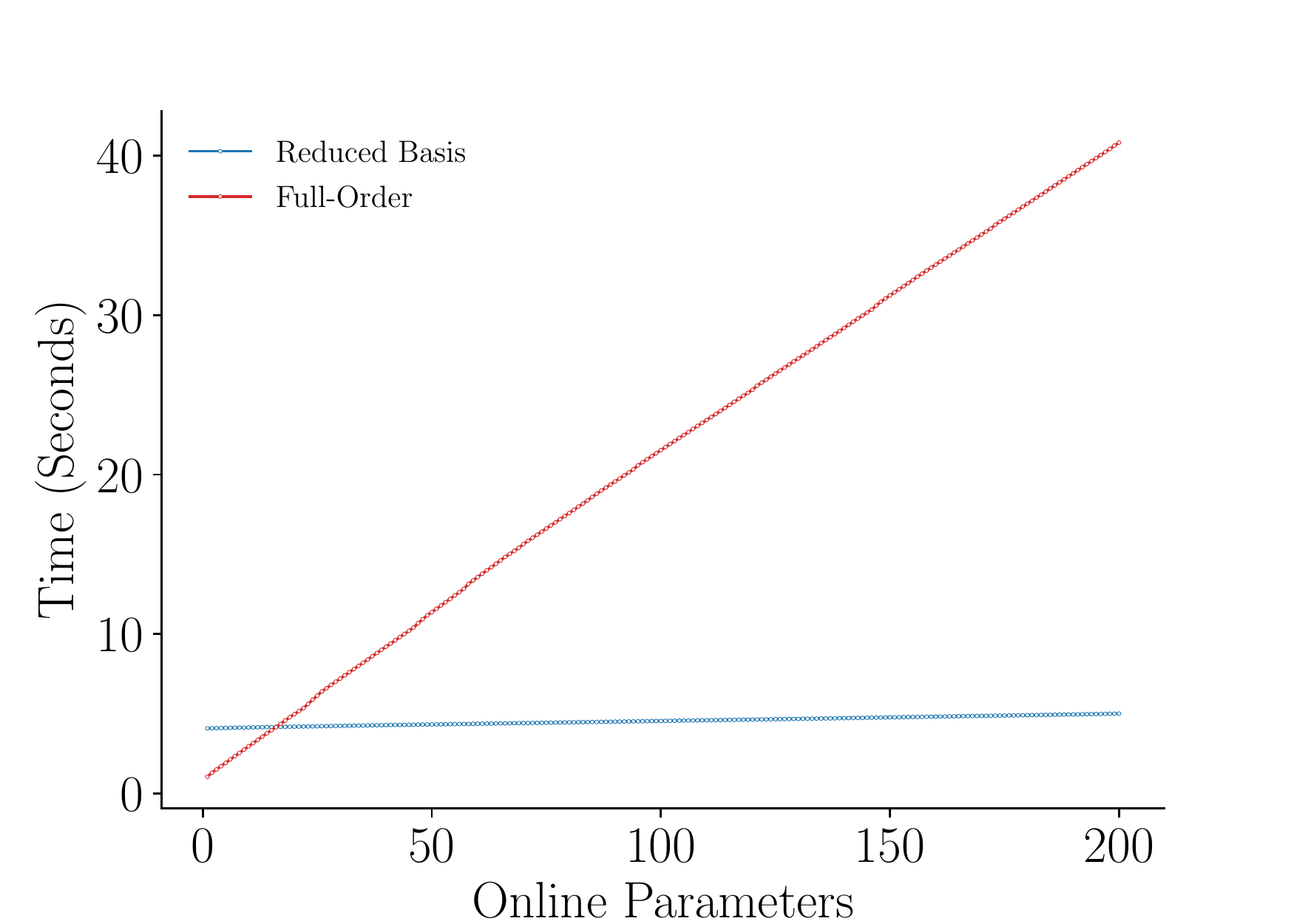}
  \caption{The total run-time (offline and online) of the reduced basis method for each new parameter value encountered in the online stage compared to the use of full-order models.  The reduced basis method incurs an offline cost of approximately 4 seconds, but quickly becomes an extremely efficient alternative over the use of full-order models.}\label{fig:thermal-1-run_time}
\end{figure}

\subsection{Thermal Block~---~3 Parameters}\label{sec:thermal3}

We repeat the same variable coefficient Poisson problem, now with four subdomains, and consequently, three different parameter values $\vmu = (\mu_1,\mu_2, \mu_3)^T$.
\begin{figure}[ht]
\begin{center}
\begin{tikzpicture}
\draw (0,0) rectangle (4,4);
\draw (2,0) -- (2,4);
\draw (0,2) -- (4,2);
\node at (1,1) {$\kappa(\bm{x}) = \mu_1$};
\node at (3,1) {$\kappa(\bm{x}) = \mu_2$};
\node at (1,3) {$\kappa(\bm{x}) = \mu_3$};
\node at (3,3) {$\kappa(\bm{x}) = 1$};
\end{tikzpicture}
\caption{Conductivity for the variable coefficient Poisson equation with four subdomains.  The parameter $\mu$ takes values in the interval $[5^{-1}, 5^1]$.}\label{three_parameter_domain}
\end{center}
\end{figure}

The flux reaches a singularity $(1/2, 1/2)$, where all four subdomains meet, which necessitates a finer grid and requires computing the auxiliary error equation by performing a mesh refinement in addition to the increase in polynomial order.  Our approximation is again computed on $X^h = (\text{RT}_0) \times P_1$, with 1,456 degrees of freedom.  The auxiliary error equation is computed after one mesh refinement using $(\text{RT}_1) \times P_2$ elements, which corresponds to 20,384 degrees of freedom.

The reduced basis is constructed using a sample of 75 randomly generated samples $\vmu \in [0.2, 5.]^3$ using Latin hypercube sampling.
We also include the vertices of the parameter domain cube.
The SCM method requires the solution of 45 generalized eigenvalue problems of size $1,456\times 1,456$, and the reduced basis construction terminates after the computation of $N = 13$ basis functions with a final tolerance of $\delta \approx 0.7557$.
Thus, the error estimate is guaranteed to satisfy the effectivity bound
\begin{align}\label{thermal-3-effectivity-bound}
	\frac{M^N(\mu)}{\|e_{\mu}^N\|_X} \leq \frac{1 + \delta}{1 - \delta} \approx 7.1877.
\end{align}

To test the reduced basis, we generate 100 randomly sampled parameter values $\mu \in [0.2, 5.0]^3$ using Latin hypercube sampling and compute a reference solution using $(\text{RT}_2)\times P_3$ elements after performing two uniform mesh refinements.
This corresponds to 174,720 degrees of freedom.  We then compute the reduced basis approximation for these parameter values and the corresponding RB error estimate.  For each parameter, $\alpha_{\text{LB}}(\vmu)$ is computed through the online SCM algorithm; the resulting linear program has 7 variables and 39 inequality constraints.  The reduced basis solution and approximate error requires the solution of two $13\times 13$ linear systems.  The true error and the corresponding error estimate are shown for the test set in Figure~\ref{thermal-3-error}.  Once again, we see a rigorous upper bound of the error over the testing set.\par 

\begin{figure}[ht]
\centering
\begin{subfigure}[b]{.48\textwidth}
  \vspace{0pt}
	\centering
  \includegraphics[height=1.76in]{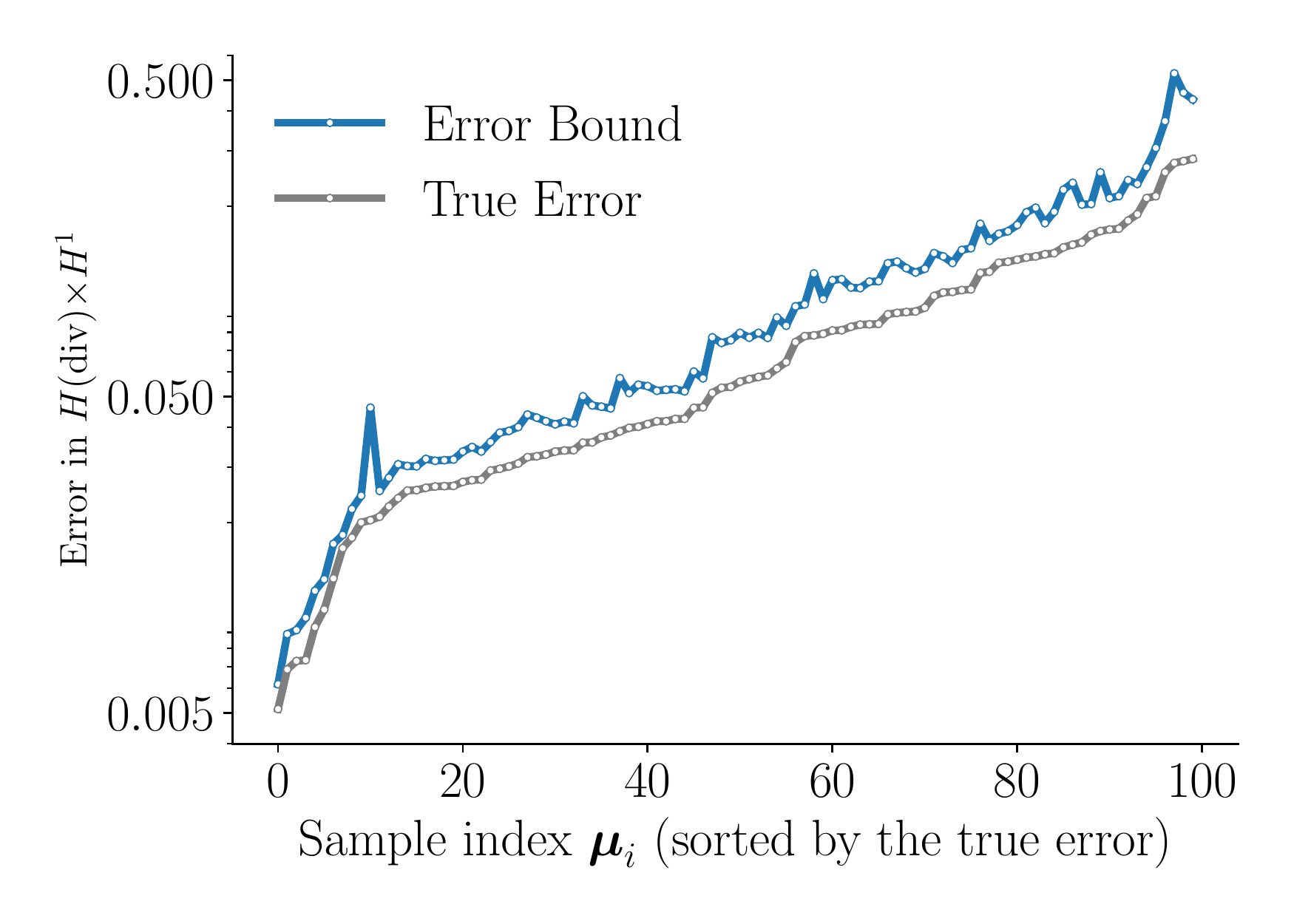}
  \caption{}\label{thermal-3-error}
\end{subfigure}\hfill
\begin{subfigure}[b]{.48\textwidth}
  \vspace{0pt}
	\centering
	\includegraphics[height=1.76in]{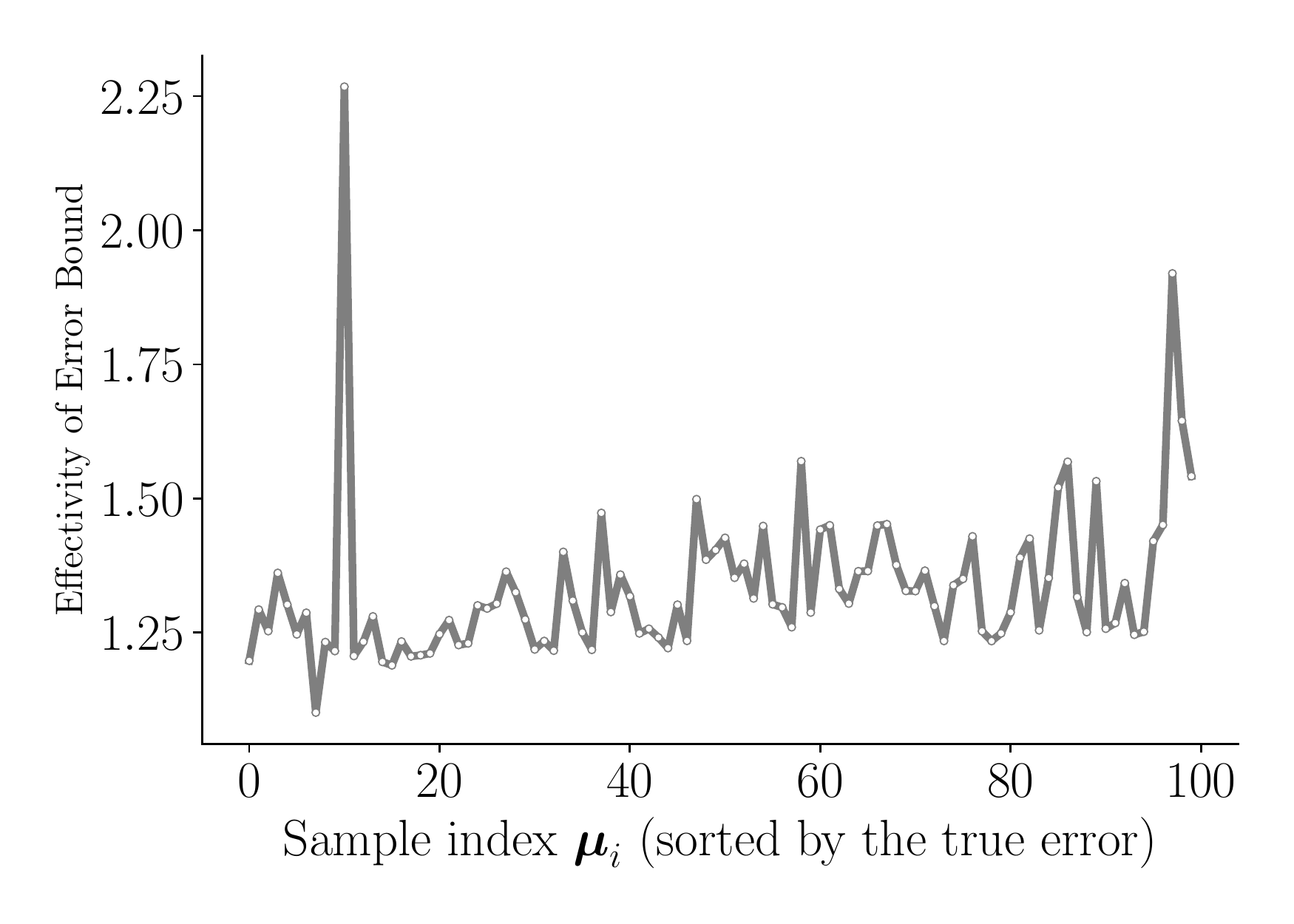}
	\caption{}\label{thermal-3-effectivity}
\end{subfigure}
\caption{(a) The error between the RB solution and the reference solution, along with the corresponding RB error bound over the testing set of parameter values of $\mu$ for the three parameter thermal block.  (b) The effectivity of the RB error estimate in~\eqref{thermal-1-effectivity-bound} over the testing set of parameter values in the three parameter thermal block.}
\end{figure}

Plotting the effectivity of the parameter set in Figure~\ref{thermal-3-effectivity}, we see that the effectivity is bounded by 1.5 over much of the testing set, and the error bound overestimates the true error by no more than a factor $< 2.4
$.  In this case the guaranteed effectivity bound~\eqref{thermal-3-effectivity-bound} is a slightly pessimistic prediction on the tightness of the error bound.\par
We also examine the convergence of the reduced solution as the dimension of the basis increases.  We choose the parameter value in the testing set with the largest effectivity, i.e. the one corresponding to the largest over-estimation of the error via the reduced basis method with $N = 13$.  In this case, $\vmu \approx (.223, .244, .746)$ and the error is overestimated by a factor of approximately 2.26.  The true error and corresponding error estimate is shown in Figure \ref{fig:error_conv} as a function of basis dimension.  The error estimate remains rigorous and reliable for all sizes of basis; the worst effectivity is approximately $3.76$ corresponding to $N = 3$.

\begin{figure}[ht]
\centering
	\includegraphics[scale=0.50]{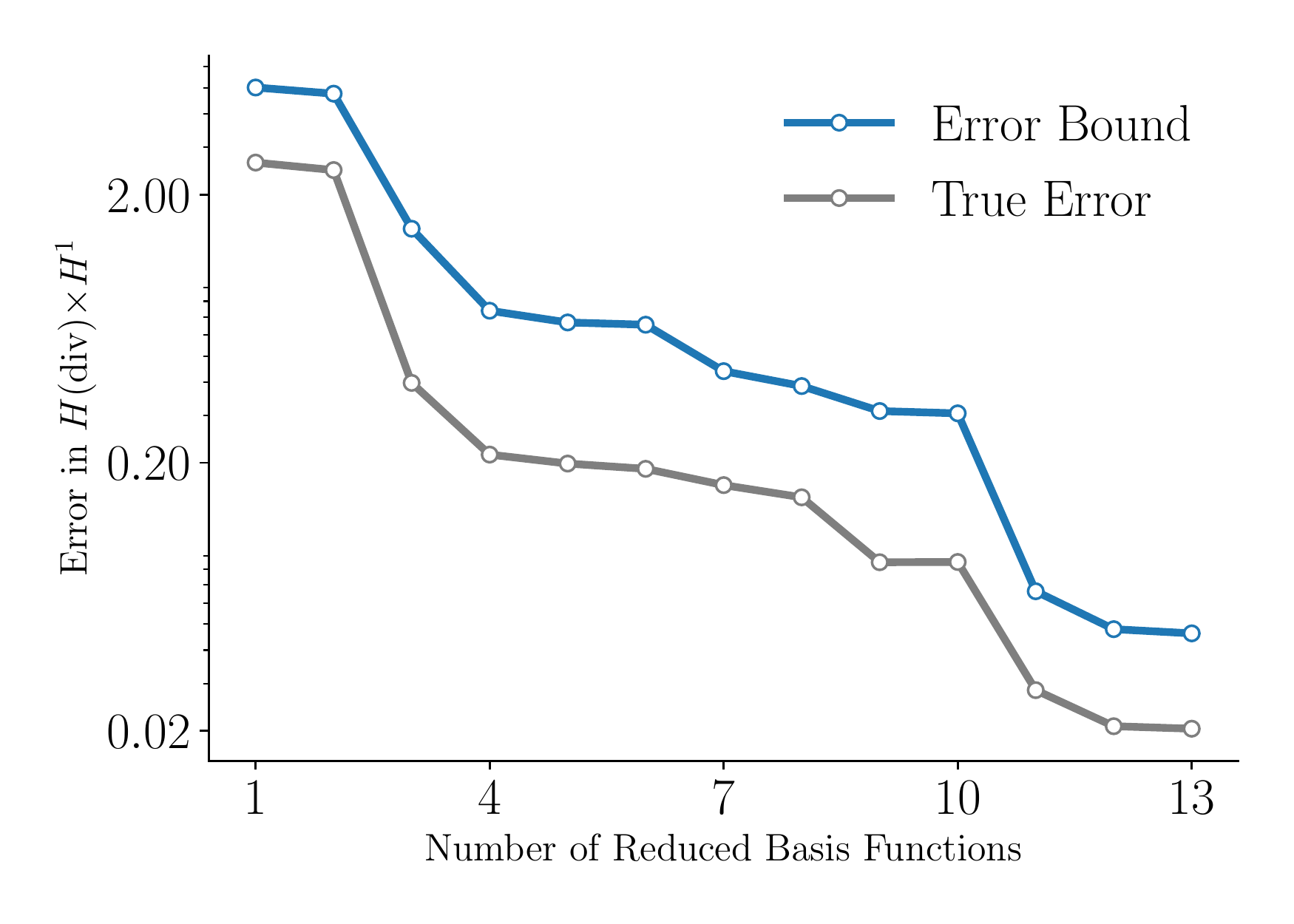} 
  \caption{Convergence of the error bound and true error for increasing RB basis dimension $N$.  The parameter value is $\vmu \approx (.223, .244, .746)$, corresponding to the largest over-estimation of the error.}\label{fig:error_conv}
\end{figure}

In terms of run-time, the online algorithm takes considerably longer than the 1 parameter thermal block problem; approximately 43.5 seconds.  However, the average run-time per parameter for the RB solution and error estimate in the online stage is only around $5.4\times 10^{-3}$ seconds.  In comparison, the full-order solution and error estimate takes $1.4$ seconds per parameter on average.  The speed-up is over 250 times faster and the RB approach is more efficient for 29 or more parameters in the online stage.

\subsection{Linear Elasticity}\label{sec:elasticity}

For this experiment we consider linear elasticity and the model problem originating from~\cite{ramm2003error}.  The setup consists of a two-dimensional plate with a circular whole at the center.  Given the symmetry of the problem we consider only the upper right quarter for $\Omega$ as in Figure~\ref{fig:elasticity-domain}.
\begin{figure}[ht]
\centering
\begin{subfigure}[b]{.49\textwidth}
  \vspace{0pt}
	\centering
  \includegraphics[height=2in]{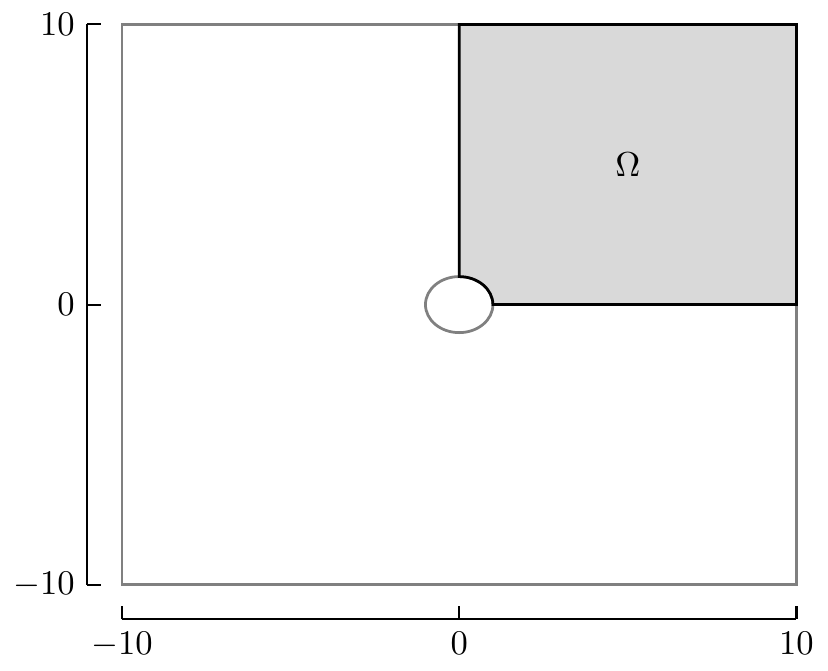}
  \caption{Domain $\Omega$.}\label{fig:elasticity-domain}
\end{subfigure}\hfill
\begin{subfigure}[b]{.49\textwidth}
  \vspace{0pt}
	\centering
	\includegraphics[height=1.5in]{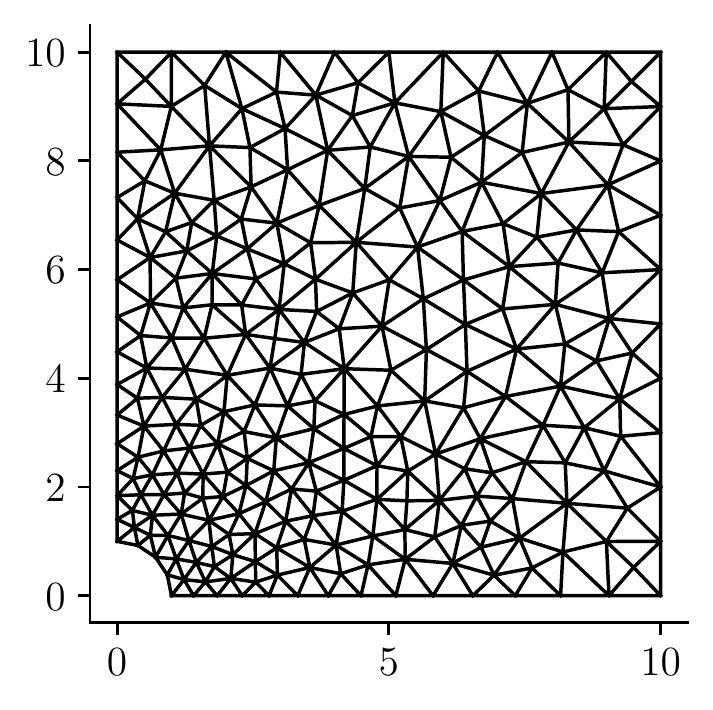}
	\caption{Mesh with 466 cells.}\label{elasticity_mesh}
\end{subfigure}
\caption{Domain and mesh for the elasticity problem.}\label{fig:elasticity-setup}
\end{figure}

We denote the material properties as
$E$ (Young's Modulus) and $\nu$ (Poisson's ratio), which are related through the  Lam\'e constants
\begin{equation}
	\lambda = \frac{E\nu}{(1 + \nu)(1-2\nu)},\hspace{8mm} \mu = \frac{E}{2(1 + \nu)}.
\end{equation}
Then let $\vu = [u_1, u_2]^T$ be the displacement of the plate, and let $\vsigma$ be the $2\times 2$ stress tensor.  Using a substitution $\vsigma \leftarrow \frac{1}{\mu}\vsigma$, leads to a change in units so that $\mu = 1$.  With this we arrive at the following first-order system of PDEs, following~\cite{cai2004least}:
\begin{align}\label{elasticity_system_homogeneous}
\begin{split}
	\mathcal{A}\vsigma - \vepsilon(\vu) &= \mathbf{0},\\
	\nabla \cdot \vsigma &= \mathbf{0}.
	\end{split}
\end{align}
Here, we assumed no body forces, the divergence of a tensor is taken row-wise, and the operators $\mathcal{A}$ and $\vepsilon$ are defined as:
\begin{align}
	\begin{split}
		\mathcal{A}\vsigma &= \frac{1}{2}\left(\vsigma - \frac{\lambda}{2(\lambda + 1)}(\text{tr}\vsigma)\vI\right) = \frac{1}{2}\left(\vsigma - \nu(\text{tr}\vsigma)\vI\right) \\
		\epsilon(\vu) &= \frac{1}{2}\left(\nabla\vu + \nabla\vu^T\right).
	\end{split}
\end{align}
We apply a (scaled by $\mu$) traction force via the boundary condition $\vsigma\vn = K\vn$, along the top boundary $y = 10$.  To enforce this inhomogenous boundary condition, a lifting function $\vsigma_{\ell}$ is introduced that satisfies this condition.

The parameters for this problem are now of the form $\vmu = [\mu_1,\mu_2]^T = [\nu, K]^T$.  We restrict Poisson's ratio $\nu$ to values in $[0.1, 0.5]$, since $0.5$ corresponds to an incompressible material. In addition, we limit the scaled traction coefficient $K$ to the interval $[-0.25, 0.25]$.

Each row of $\vsigma$ is viewed as a two-dimensional vector, and we define an operator $\mathcal{L}_{\vmu} = \mathcal{L}_{\nu}$ that maps $\vU = [\vsigma, \vu] \in X \subset \left[H(\text{div};\Omega)\right]^2 \times \left[H^1(\Omega)\right]^2$ into $Y = \left[L^2(\Omega)\right]^{2\times 2} \times \left[L^2(\Omega)\right]^2$:
\begin{equation}
	\mathcal{L}_{\nu}\vU = \begin{pmatrix}
		\mathcal{A} & - \epsilon\\
		\nabla \cdot & 0\end{pmatrix}\begin{pmatrix}
			\vsigma\\
			\vu
		\end{pmatrix} =
		\begin{pmatrix}
			\mathbf{-\mathcal{A}\vsigma_{\ell}}\\
			\mathbf{-\nabla \cdot \vsigma_{\ell}}
		\end{pmatrix}
\end{equation}
Here, $X$ is the subspace of functions that satisfy the corresponding homogeneous boundary conditions.  This form of $\mathcal{L}_{\nu}$ satisfies~\eqref{energy_balance} (see~\cite{cai2004least})
with respect to the norm
\begin{equation}
	\|(\vtau,\vv)\|_X^2 = \|\vepsilon(\vv)\|_0^2 + \|\vtau\|_0^2 + \|\nabla \cdot \vtau\|_0^2,
\end{equation}
where $\|\cdot\|_0$ is the $L^2(\Omega)$ norm for vector or tensor valued functions, depending on context.

We compute a discrete approximation using the subspace $X^h = \left(\text{RT}_0\right)^2 \times \left(\text{P}_1\right)^2$~---~i.e., approximate the rows of the stress tensor $\vsigma$ by functions in the lowest-order Raviart-Thomas space~\cite{raviart1977primal}, and the components of the displacement $\vu$ by piecewise linear polynomials.  For the mesh in Figure~\ref{elasticity_mesh}, this corresponds to $1,970$ degrees of freedom.

For the reference solutions $\vsigma_{\vmu}$ and $\vu_{\vmu}$, we perform one refinement on the original mesh, and approximate the solution by functions in $\left(\text{RT}_2\right)^2 \times \left(\text{P}_3\right)^2$, which corresponds to $56,546$ degrees of freedom.
We compute the lower bound to the coercivity constant via the SCM method, with a tolerance of 30\%.  Since the operator $\mathcal{L}_{\nu}$ only depends on Poisson's ratio, the SCM method is performed using 50 uniformly sample values of $\nu \in [0.1, 0.5]$.

For the basis construction, $\mathcal{D}_{\text{train}}$ consists of a $10\times 10$ uniform grid sampling of $(\nu, K) \in [0.1,0.5]\times [-0.25, 0.25]$.   Algorithm~\ref{ls-algorithm} terminates after computing $N = 5$ basis functions with a final tolerance of $\delta \approx 0.6445$.  That is, all reduced-order solutions $\vsigma_{\vmu}^N$, $\vu_{\vmu}^N$ corresponding to parameters in the sampled grid satisfy

\begin{align}\label{elasticity_effectivity}
	\begin{split}
		\frac{\|\rho_{\vmu}^N\|_0}{\sqrt{\alpha_{\text{LB}}(\vmu)}\|\hat{e}_{\vmu}^N\|_X} \leq \delta \approx 0.6445\\
		\frac{M^N(\vmu)}{\|e_{\vmu}^N\|_X} \leq \frac{1 + \delta}{1 - \delta} \approx 4.6266
	\end{split}
\end{align}
so that our error bound overestimates the true error by at most a factor of $4.6266\times$.

To test the reduced basis, we generate 100 randomly sampled $(\nu, K)$ pairs in $[0.1,0.5]\times [-0.25, 0.25]$ that were not involved in the basis construction.  In Figure~\ref{elasticity_error}, we see that the error bound generated by the reduced basis approximation is a rigorous bound for all parameters in the testing set.

\begin{figure}[ht]
\centering
\begin{subfigure}[b]{.48\textwidth}
  \vspace{0pt}
	\centering
  \includegraphics[height=1.72in]{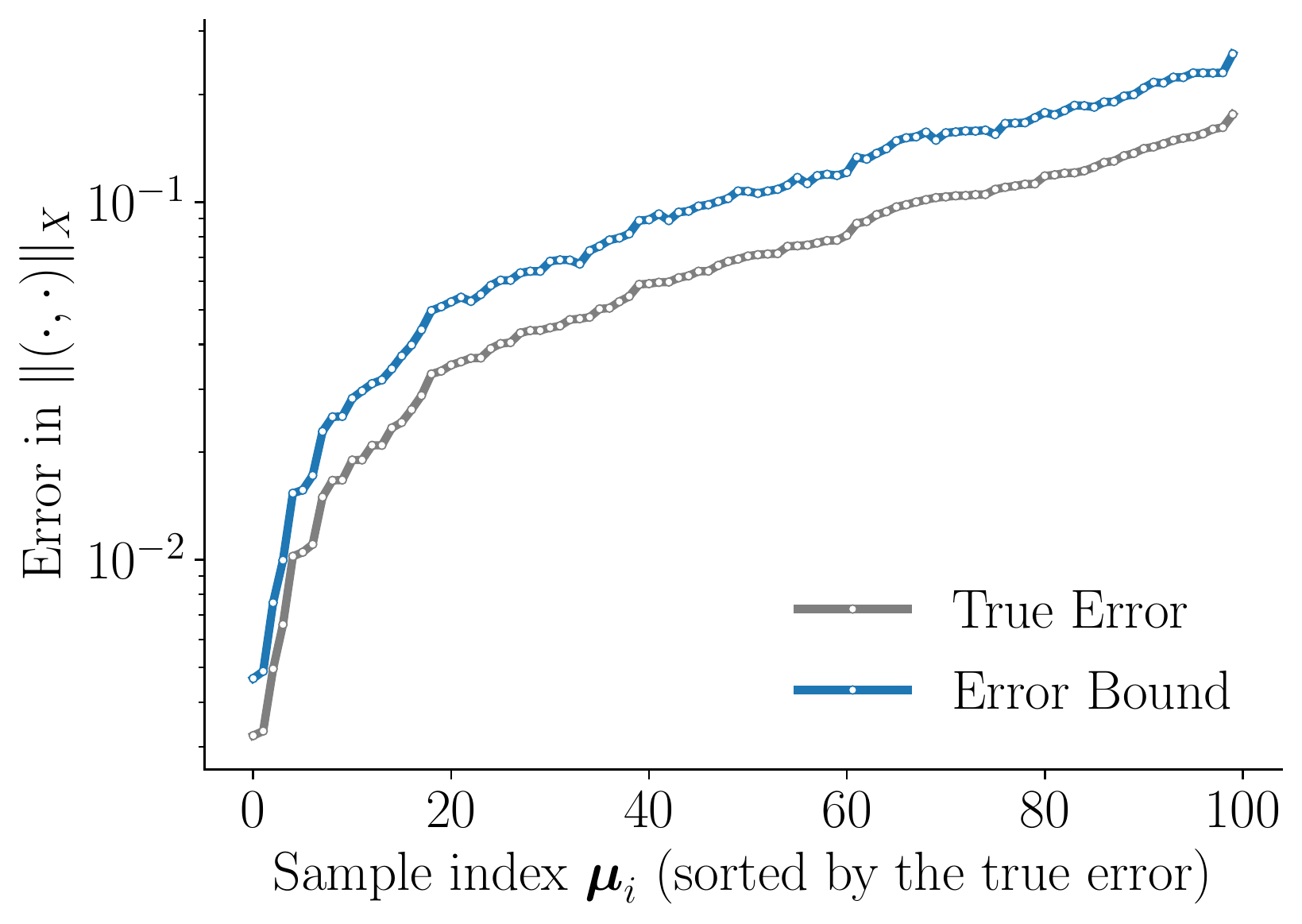}
  \caption{}\label{elasticity_error}
\end{subfigure}\hfill
\begin{subfigure}[b]{.48\textwidth}
  \vspace{0pt}
	\centering
	\includegraphics[height=1.72in]{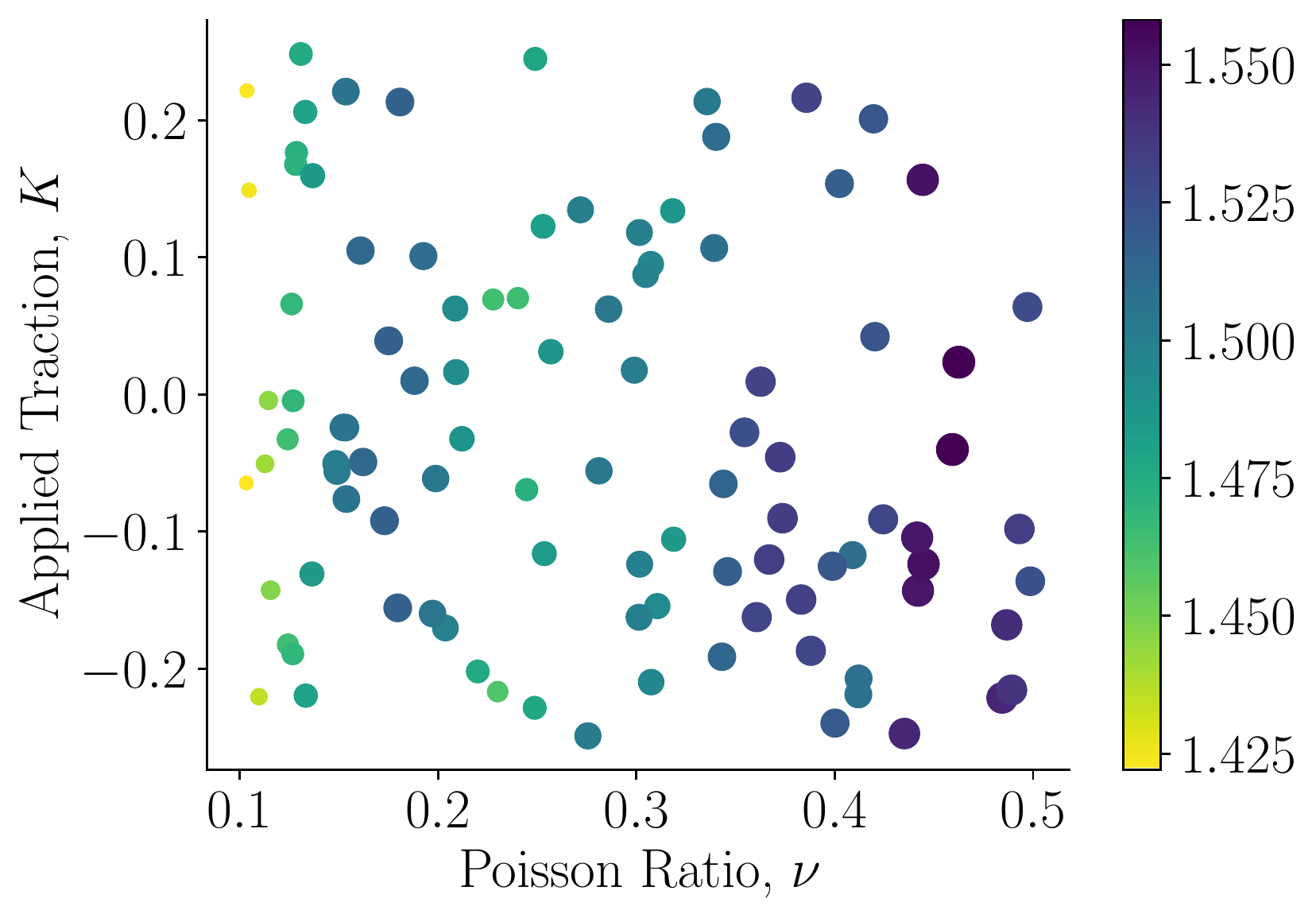}
	\caption{}\label{elasticity_effectivity_plot}
\end{subfigure}
\caption{(a) Error over 100 samples of $(\nu,K)$, computed with respect to a high-order representation of smooth solution, labeled ``True Error''. The error bounds generated by reduced basis solution is labeled as ``Error Bound''.  Note: The $x$-axis corresponds to the indices of the test parameters, which are ordered by magnitude of the true error.  (b) The effectivity ratios $\frac{M^N(\vmu)}{\|e_{\vmu}^N\|_X}$ over the test parameter set.}
\end{figure}

We next examine the bound on the effectivity ratio in~\eqref{elasticity_effectivity}, which is again pessimistic.  Indeed, the mean effectivity over the test set is approximately $1.501$ and no error bound has an effectivity larger than $1.558$, as shown in Figure~\ref{elasticity_effectivity_plot}.
\section{Conclusions and Future directions}\label{sec:conc}

In this paper we have introduced
a reduced basis method for parametrized elliptic partial differential equations using least-squares finite element methods.  We demonstrated that the first-order system formulation provides an opportunity to construct a rigorous error bound on the \textit{exact} solution by solving an auxiliary error problem.  This is in contrast to standard RB approaches that estimate the error with respect to a solution from a fixed finite-dimensional subspace.  Rigorous bounds on the effectivity of this estimate have also been established when the auxiliary equation properly resolves the error.

The least-squares finite element reduced basis method is also applicable to bases constructed via POD\@.  In the offline stage, the effectivity of the error bound no longer guides the sampling of the parameter domain since the POD algorithm relies on the decay of the eigenvalues to form the basis. However, the decay of the eigenvalues does not give a quantitative bound on the actual error of reduced basis approximations. The error estimate and the bound on the effectivity developed in this article may be used  after the POD basis is formed to  give an indication of whether the basis was truncated too soon. This in turn  should provide guidance in determining the number of basis functions needed to produce sufficiently accurate reduced basis solutions.

From the numerical experiments, we see that the bound on the effectivity, while not sharp, is still accurate.  Since the error of the RB solution is estimated with respect to the true solution, there may be regions of the parameter domain that require much finer mesh resolution or polynomial orders.  Using this reduced basis method as a guide to partitioning the parameter domain into separate reduced order models has the potential to increase accuracy and develop sharper effectivity bounds.

In many cases, an output or quantity of interest $Q(u_{\vmu})$ is of more interest than the solution itself.  Future work should  extend the least-squares finite element reduced basis method to these situations by developing computable bounds on the error $|Q(u_{\vmu}) - Q(u_{\vmu}^N)|$.\par
Finally, least-squares finite element methods are not the only variational method that re-formulates a PDE into a first-order system.  An investigation of other such methods, e.g., mixed Galerkin finite element methods, would make the results more broadly applicable.

\bibliographystyle{siamplain}
\bibliography{refs-romls}
\end{document}